\documentclass[12pt]{article}
\usepackage{amsthm,amsfonts,amssymb,amsmath}
\usepackage{stmaryrd}
\usepackage{cite,hyperref}
\usepackage{epsfig}
\usepackage{url}
\usepackage{xcolor,tikz}
\usetikzlibrary{positioning, arrows.meta,calc}
\usetikzlibrary{matrix}
\usepackage{multicol,graphicx}
\usepackage{fullpage}
\usepackage{bbm}
\usetikzlibrary{decorations}
\usepackage{svg}

\usepackage[shortlabels]{enumitem}

\numberwithin{equation}{section}

\newtheorem{thm}[equation]{Theorem}

\newtheorem{lem}[equation]{Lemma}

\newtheorem{conj}[equation]{Conjecture}

\theoremstyle{definition}
\newtheorem{defn}[equation]{Definition}

\newtheorem{const}[equation]{Construction}
\newtheorem*{ack}{Acknowledgements}

\theoremstyle{remark}
\newtheorem{case}{Case}
\newtheorem{casee}{Case}
\newtheorem{caseee}{Case}

\title{Optimal and Near-Optimal Constructions for Bootstrap Percolation in Hypercubes}
\author{Jonathan A. Noel\thanks{Department of Mathematics and Statistics, University of Victoria, Victoria, B.C., Canada. E-mail: {\tt noelj@uvic.ca}. Research supported by NSERC Discovery Grant RGPIN-2021-02460.}}

\DeclareTextCompositeCommand{\v}{OT1}{l}{l\nobreak\hspace{-.1em}'}
\DeclareTextCompositeCommand{\v}{OT1}{t}{t\nobreak\hspace{-.1em}'\nobreak\hspace{-.15em}}

\begin{document}

\maketitle

\begin{abstract}
The \emph{$r$-neighbour bootstrap process} on a graph $G$ begins with a set of infected vertices; subsequently, healthy vertices become infected once they have at least $r$ infected neighbours. The central extremal problem in bootstrap percolation is to determine the minimum cardinality of an initial infected set that eventually spreads to all vertices of $G$, denoted $m(G;r)$. Morrison and Noel~\cite{MorrisonNoel18} established a general lower bound on $m(Q_d;r)$, where $Q_d$ is the $d$-dimensional hypercube, and asked whether it is tight whenever $d$ is sufficiently large with respect to $r$. This question was answered affirmatively for $r\leq 3$. In this paper, we show that $m(Q_d;4)=\frac{d(d^2+3d+14)}{24}+1$, matching the bound in~\cite{MorrisonNoel18}, for infinitely many $d$. We also obtain, for general $d$, an upper bound on $m(Q_d;4)$ that differs from the Morrison--Noel lower bound by an additive $O(d)$ term. Several key constructions in this paper were obtained with the assistance of AlphaEvolve~\cite{Novikov25}.
\end{abstract}

\section{Introduction}

The \emph{$r$-neighbour bootstrap process} on a graph $G$ starts with a set $A_0$ of infected vertices (with all other vertices initially healthy) and, in each round, a healthy vertex becomes infected if it has at least $r$ infected neighbours. That is, if $A_0$ is the set of infected vertices at time zero, then 
\[A_t:=A_{t-1}\cup\{v\in V(G): |N_G(v)\cap A_{t-1}|\geq r\}\]
is the set of vertices infected after time $t$ for any $t\geq1$, where $N_G(v)$ is the neighbourhood of $v$ in $G$. We say that $A_0$ \emph{percolates} if $\bigcup_{t=0}^\infty A_t=V(G)$.

Despite its simple definition, the bootstrap process leads to many beautiful and challenging mathematical problems. One major line of research concerns estimating the \emph{critical probability} for bootstrap percolation, defined as the infimum over all $p$ for which a random set of density $p$ percolates with probability at least $1/2$, in square grids of various dimensions~\cite{BaloghBollobas06,vanEnter87,BaloghBollobasDCMorris12,AizenmanLebowitz88,BaloghBollobasMorris09a,BaloghBollobasMorris09b,BaloghBollobasMorris10,CerfCirillo99,CerfManzo02,GravnerHolroydMorris12,Holroyd03,HartarskyMorris19,Uzzell19}.

Our focus is on the well-studied extremal problem of determining the minimum cardinality of a percolating set under the $r$-neighbour bootstrap process in a graph $G$, which we denote by $m(G;r)$; see, e.g.,~\cite{MiralaeiMohammadianTR26,HambardzumyanHatamiQian20,DukesNoelRomer23,BenevidesBermondLesfariNisse24,PrzykuckiShelton20,Pete97}. A well-known exercise which appears in many puzzle books, e.g.~\cite[Problem~34]{Bollobas06} and~\cite[p.~79]{Winkler04}, asks one to prove that
\[m([n]^2;2)=n\]
where $[n]^d$ denotes the $n\times n\times\cdots \times n$ grid in $d$ dimensions.\footnote{If you do not know the solution, we recommend pausing here to enjoy the puzzle before reading further.} According to~\cite{BaloghMR}, it originated in the Soviet popular science magazine \emph{Kvant} in the mid-1980s. It was later brought to Szeged by Lajos Pint\'er who encouraged G\'abor Pete to write an undergraduate thesis on it~\cite{Pete97}, and it has continued to percolate\footnote{Pun intended.} throughout  mathematics ever since.

In this paper, we focus on the \emph{$d$-dimensional hypercube} $Q_d$; i.e. the graph with vertex set $\{0,1\}^d$ in which two vertices are adjacent if they differ in exactly one coordinate. Balogh and Bollob\'as~\cite[Lemma~7(ii)]{BaloghBollobas06} proved that 
\begin{equation}
\label{eq:r=2}
m(Q_d;2)=\lceil d/2\rceil+1
\end{equation}
for all $d\geq2$ and conjectured\footnote{While this conjecture does not actually seem to appear in~\cite{BaloghBollobas06}, a later paper of Balogh, Bollob\'as, Morris and Riordan~\cite{BaloghBollobasMorrisRiordan12} attributes it to~\cite{BaloghBollobas06}. A weaker conjecture also appears in~\cite{BaloghBollobasMorris10}.} that
\begin{equation}
\label{eq:BBConj}
m(Q_d;r)=(1+o(1))\frac{d^{r-1}}{r!}
\end{equation}
for each fixed $r\geq1$ and large $d$. The following theorem of Morrison and Noel~\cite{MorrisonNoel18} settled the aforementioned conjecture of Balogh and Bollob\'as~\cite{BaloghBollobas06}.

\begin{thm}[Morrison and Noel~\cite{MorrisonNoel18}]
\label{th:MN}
For $d\geq r\geq1$,
\[m(Q_d;r)\geq 2^{r-1}+\sum_{j=1}^{r-1}\binom{d-j-1}{r-j}\frac{j2^{j-1}}{r}.\]
\end{thm}

In this paper, we focus on $m(Q_d;4)$. Our main result shows that the lower bound in Theorem~\ref{th:MN} is tight for infinitely many dimensions $d$.

\begin{thm}
\label{th:main}
For every $d\geq4$ such that either $d=4$, $6\leq d\leq 15$ or $d\equiv 0,4\bmod 6$,
\[m(Q_d;4) = \frac{d(d^2+3d+14)}{24}+ 1.\]
\end{thm}

We also obtain constructions within only a linear difference from the lower bound in Theorem~\ref{th:MN} for all dimensions $d\geq4$. We did not attempt to optimize the constant in the linear term; it may be possible to reduce it via a more careful analysis.

\begin{thm}
\label{th:linearDiff}
We have
\[m(Q_5;4)\leq 14\]
and, for all $d\geq6$,
\[m(Q_d;4) \leq \left\lceil\frac{d(d^2+3d+14)}{24}\right\rceil + 1 + 20\left\lfloor \frac{d-4}{12}\right\rfloor.\]
\end{thm}

The high-level strategy for proving our main theorems is to begin with a few ad hoc constructions in small dimensions and use them to build optimal $4$-neighbour percolating sets in higher dimensions. The key idea is to partition $Q_d$ into a ``cube of subcubes'' and place optimal percolating sets for the $r$-neighbour bootstrap process, with $r\in\{1,2,3,4\}$, into carefully chosen subcubes. Because we have only a very small budget of initially infected vertices, these placements must be chosen carefully to ensure that the infection spreads as efficiently as possible. We also require that the constructions for $r=2$ are nested within those for $r=3$. Consequently, in addition to proving our results on the case $r=4$, we also reprove the result of~\cite[Theorem~1.4]{MorrisonNoel18} for $r=3$ in a form that ensures the required nesting of the constructions.

We developed the overall strategy described above by extending ideas in~\cite[Section~5]{MorrisonNoel18}, but the specific ``building block'' constructions\footnote{Specifically, the constructions in Sections~\ref{sec:constructions3},~\ref{sec:constructionsSeed4} and~\ref{sec:constructionsMeta4}.} needed to implement it were discovered with the assistance of the AlphaEvolve~\cite{Novikov25} system developed by Google DeepMind. AlphaEvolve explored a wide range of search heuristics autonomously, without requiring us to suggest specific approaches or supervise the search process. Our role was to formulate the global strategy, guide the search via prompts, and generate code (see this \href{https://colab.research.google.com/drive/1_owaZNVjQTtzwKwdJ38Z51VGLKjKb-Fd?usp=sharing}{Google Colab Notebook} and the ancillary file \texttt{FourNeighbourBootstrapVerification.ipynb} uploaded with the arxiv version of this paper) to independently verify all constructions produced by AlphaEvolve. 

In the next subsection, we provide context by discussing some of the existing constructions in the literature and how they compare to the lower bound in Theorem~\ref{th:MN}. In Section~\ref{sec:strategy}, we outline our main approach and describe the types of constructions that we require to prove Theorems~\ref{th:main} and~\ref{th:linearDiff}. The main constructions, together with proofs that they have the required properties, are provided in Sections~\ref{sec:constructions3},~\ref{sec:constructionsSeed4} and~\ref{sec:constructionsMeta4}. We put all of these pieces together to derive the main theorems in Section~\ref{sec:proofs}.

\subsection{Prior Constructions}
\label{sec:pastWork}

The purpose of this subsection is to compare Theorems~\ref{th:main} and~\ref{th:linearDiff} to the existing upper bounds in the area. The conjecture \eqref{eq:BBConj} of~\cite{BaloghBollobas06} was inspired by the following construction (see~\cite[p.~1330]{BaloghBollobasMorrisRiordan12}).

\begin{const}
\label{const:Steiner}
Let $d\geq2r-3$ and take $A_0\subseteq V(Q_d)$ to consist of all vertices whose coordinates sum to $r-2$ (i.e.\ vertices on \emph{layer} $r-2$), together with a set of $\frac{1+o(1)}{r}\binom{d}{r-1}$ vertices on layer $r$ whose neighbourhoods cover all vertices on layer $r-1$. The existence of such a set follows from R\"odl's theorem on approximate Steiner systems~\cite{Rodl85}. Every vertex on layer $r-1$ has $r$ neighbours in $A_0$ and is therefore infected after one step. Once layers $r-2$ and $r-1$ are all infected, the infection will spread to layers $r,r+1,\dots$ as well as layers $r-3,r-4,\dots$, sequentially. In the best case, i.e. when an exact Steiner system exists, this construction has cardinality $\frac{1}{r}\binom{d}{r-1}+\binom{d}{r-2}$, which is
\begin{equation}
\label{eq:SteinerUpper}
\frac{d^{r-1}}{r!} + \left(\frac{r+2}{2r(r-2)!}\right)d^{r-2} + O(d^{r-3}).
\end{equation}
\end{const}

By expanding the right side of the lower bound in Theorem~\ref{th:MN}, we see that it yields
\begin{equation}
\label{eq:MN}
m(Q_d;r)\geq \frac{d^{r-1}}{r!} + \left(\frac{6-r}{2r(r-2)!}\right)d^{r-2} + O(d^{r-3}).
\end{equation}
For $r=2$, both the first and second terms of \eqref{eq:MN} match with \eqref{eq:SteinerUpper} and with the exact result \eqref{eq:r=2}. However, for $r\geq3$, there is a gap between the second-order terms of \eqref{eq:MN} and \eqref{eq:SteinerUpper}. In the case $r=3$,  Morrison and Noel~\cite{MorrisonNoel18} provided an upper bound construction which showed that Theorem~\ref{th:MN} is tight; in particular
\begin{equation}\label{eq:r=3}m(Q_d;3)=\left\lceil \frac{d(d+3)}{6}\right\rceil +1\end{equation}
for all $d\geq3$. This led them to ask whether, for all $r\geq2$, there exists $d_0(r)$ such that the lower bound in Theorem~\ref{th:MN} is tight for all $d\geq d_0(r)$. B\'erczi and Wagner~\cite{BercziWagner24} recently provided the following upper bound construction which improves Construction~\ref{const:Steiner} and gets close to the lower bound in \eqref{eq:MN} for $d$ and $r$ satisfying certain conditions. 

\begin{const}
\label{const:BW}
For $r\geq3$, let $d\geq 2r-3$ such that there exists a set of vertices of cardinality $\frac{1}{r}\binom{d}{r-1}$ on layer $r$ of $Q_d$ whose neighbourhoods cover every vertex on layer $r-1$ exactly once. Take $A_0$ to consist of all vertices of the aforementioned set on layer $r$ together with all vertices of $Q_d$ on layer $r-3$ and a set of $\frac{2+o(1)}{r-1}\binom{d}{r-2}$ vertices on layer $r-1$ whose neighbourhoods cover every vertex on layer $r-2$ at least twice; again, such a set exists by the result of~\cite{Rodl85}. The cardinality is 
\begin{equation}
\label{eq:BW}
\frac{d^{r-1}}{r!} + \left(\frac{6-r}{2r(r-2)!} + \frac{2}{r!}\right)d^{r-2} + O(d^{r-3}).
\end{equation}
\end{const}

Thus, when $r\geq3$, Construction~\ref{const:BW} improves upon Construction~\ref{const:Steiner} in the second-order asymptotic term. However, it still does not match \eqref{eq:MN} and, like Construction~\ref{const:Steiner}, it has the disadvantage that it depends on the existence of exact Steiner systems. B\'erczi and Wagner~\cite{BercziWagner24} also used a generative AI approach known as PatternBoost~\cite{Charton+24} to search for small percolating sets in $Q_d$ for the case $r=4$. They obtained improved upper-bound constructions that match the lower bound of Theorem~\ref{th:MN} for $d=7,8,13$, come within $1$ of the bound for $d=9,11,12$ and within $2$ of the bound for $d=10$.

\section{General Strategy}
\label{sec:strategy}

We build up percolating sets for the $4$-neighbour bootstrap process on $Q_d$ by inserting optimal constructions for the $i$-neighbour bootstrap process into smaller subcubes of $Q_d$ and analyzing the interactions between them. For this, it is useful to introduce the following auxiliary propagation process.

\begin{defn}
For a graph $G$ and $r\geq1$, we define the \emph{$r$-meta bootstrap process} on $G$ as follows. We start with an initial labeling $\ell_0:V(G)\to \{0,\dots,r\}$. The labeling is then updated according to the following two rules:
\begin{itemize}
    \item \textbf{Rule 1 (Completion):} If a vertex of label $i$ has at least $r-i$ neighbours of label $r$, then its label is updated to $r$. 
    \item \textbf{Rule 2 (Promotion):} If a vertex $v$ has at least $r$ neighbours of higher label than its own, then the label of $v$ is updated to the $r$th highest label among its neighbours. 
\end{itemize}
The process terminates when no vertex satisfies the conditions of Rule 1 or Rule 2. 
\end{defn}

Observe that, due to the monotonicity of the rules, the final labeling is independent of the order in which the rules are applied. We say that a labeling $\ell_0$ \emph{percolates} under the $r$-meta bootstrap process on $G$ if every vertex eventually receives label $r$.

The next lemma describes how we apply percolating labelings for the $r$-meta bootstrap process on a hypercube of fixed dimension to obtain small percolating sets for the $r$-neighbour bootstrap process in larger hypercubes. The following notation is useful for describing the proof. For $d>k\geq 1$ and a vertex $x=(x_1,\dots,x_k)$ of $Q_k$, let $Q_d^{x}$ be the subgraph of $Q_d$ consisting of all vertices $v$ such that the projection of $v$ onto its first $k$ coordinates is equal to $x$. It is not hard to see that $Q_d^{x}$ is isomorphic to $Q_{d-k}$. Also, given a set $S\subseteq V(Q_{d-k})$ and a vertex $x$ of $Q_k$, let $S^{x}$ be the set in $V(Q_d)$ obtained by prefixing each vector in $S$ with $x$. For example, if $d=5,k=2$ and $S=\{(1,1,1),(0,1,0)\}$, then $S^{(1,0)}=\{(1,0,1,1,1),(1,0,0,1,0)\}$. The idea behind this lemma is very much inspired by the approach in~\cite[Section~5]{MorrisonNoel18}.

\begin{lem}
\label{lem:meta}
Let $r\geq2$, $k\geq1$ and $d\geq r+k$. If $S_1,S_2,\dots,S_r$ are subsets of $V(Q_{d-k})$ and $\ell_0:V(Q_k)\to\{0,1,\dots,r\}$ are such that
\begin{enumerate}[(a)]
    \item\label{eq:blocksPercolate} $S_i$ percolates with respect to the $i$-neighbour bootstrap process on $Q_{d-k}$ for all $1\leq i\leq r$, 
    \item\label{eq:labelingPercolates} $\ell_0$ percolates with respect to the $r$-meta bootstrap process on $Q_k$, and
    \item\label{eq:nested} $S_1\subseteq S_2\subseteq\cdots \subseteq S_{r-1}$,
\end{enumerate}
then the set 
\[A_0:=\bigcup_{i=1}^r\left(\bigcup_{x\in\ell_0^{-1}(i)}S_i^x\right)\]
percolates with respect to the $r$-neighbour bootstrap process on $Q_d$. Consequently,
\[m(Q_d;r)\leq \sum_{i=1}^r|\ell_0^{-1}(i)|\cdot |S_i|.\]
\end{lem}

\begin{proof}
By \ref{eq:labelingPercolates}, we can let $\ell_1,\dots,\ell_N$ be labelings of $V(Q_k)$ where $\ell_N$ maps every vertex to $r$ and, for $1\leq t\leq N$, $\ell_t$ is obtained from $\ell_{t-1}$ by applying exactly one of the rules of the $r$-meta bootstrap process on exactly one vertex of $Q_k$. 

For each $0\leq t\leq N$, define 
\[A_t^*:=\bigcup_{i=1}^{r-1}\left(\bigcup_{x\in\ell_t^{-1}(i)}S_i^x\right)\cup\left(\bigcup_{x\in\ell_t^{-1}(r)} V(Q_{d-k})^x\right).\]
Our goal is to show that, under the $r$-neighbour bootstrap process in $Q_d$,
\begin{itemize}
    \item if every vertex of $A_0$ is infected, then every vertex of $A_0^*$ eventually becomes infected, and
    \item for $0\leq t\leq N-1$, if every vertex of $A_t^*$ is infected, then every vertex of $A_{t+1}^*$ eventually becomes infected.
\end{itemize}
This is sufficient to prove the lemma because $A_N^*=V(Q_d)$. 

For the first claim, observe that, for every $x\in\ell_0^{-1}(r)$, the subgraph of $Q_d$ induced by $V(Q_{d-k})^x$ is isomorphic to $Q_{d-k}$ and so the set $S_r^x$ percolates under the $r$-neighbour bootstrap process within this set by \ref{eq:blocksPercolate}. Thus, every vertex of $A_0^*$ is eventually infected. 

Now, consider the second claim. Suppose first that $\ell_t$ and $\ell_{t+1}$ differ on one vertex $x$ whose label was changed from $i$ to $r$ via the Completion Rule. Let $y_1,\dots,y_{r-i}$ be neighbours of $x$ such that $\ell_t(y_1)=\cdots=\ell_t(y_{r-i})=r$. Then, by definition, $A_t^*$ contains every vertex of $V(Q_{d-k})^{y_1},\dots,V(Q_{d-k})^{y_{r-i}}$. Every vertex of $V(Q_{d-k})^x$ has exactly one neighbour in each of these sets and so each vertex of $V(Q_{d-k})^x$ has at least $r-i$ infected neighbours which lie outside of $V(Q_{d-k})^x$. Also, $A_t^*$ contains all vertices of $S_i^x$, which is a percolating set for the $i$-neighbour bootstrap percolation in the subgraph of $Q_d$ induced by $V(Q_{d-k})^x$. Putting this together, we get that every vertex of $V(Q_{d-k})^x$ is eventually infected. 

Now, suppose that $\ell_t$ and $\ell_{t+1}$ differ on one vertex $x$ whose label was changed from $i$ to $j$ via the Promotion Rule. Let $y_1,\dots,y_{r}$ be neighbours of $x$ such that $\ell_t(y_1),\dots,\ell_t(y_{r})\geq j$. Then, by definition of $A_t^*$ and condition \ref{eq:nested} of the lemma, we have that $A_t^*$ contains every vertex of $S_j^{y_1},\dots,S_j^{y_{r}}$. Every vertex of $S_j^x$ has exactly one neighbour in each of these sets, and so every such vertex has at least $r$ infected neighbours. Thus, every vertex of $S_j^x$ eventually gets infected. This completes the proof.
\end{proof}

Our approach in proving Theorems~\ref{th:main} and~\ref{th:linearDiff} is via an application of Lemma~\ref{lem:meta}. In order to ensure that condition \ref{eq:nested} of the lemma is satisfied, we require percolating sets for the $1$-, $2$- and $3$-neighbour bootstrap processes which are nested within one another. Constructing such sets is the goal of the next section.

\section{Nested 1-, 2- and 3-Neighbour Constructions}
\label{sec:constructions3}

The aim of this section is to prove the following lemma.

\begin{lem}
\label{lem:r=3}
If $d\geq3$, then there exist subsets $S_1^d,S_2^d$ and $S_3^d$ of $V(Q_d)$ such that
\begin{enumerate}[(a)]
\item $S_1^d\subseteq S_2^d\subseteq S_3^d$,
\item $|S_1^d|=1$,
\item $|S_2^d|=\left\lceil\frac{d}{2}\right\rceil+1$, 
\item $|S_3^d|=\left\lceil\frac{d(d+3)}{6}\right\rceil+1$, and
    \item $S_i^d$ percolates with respect to the $i$-neighbour bootstrap process on $Q_d$ for all $i\in\{1,2,3\}$. 
\end{enumerate}
\end{lem}

\begin{proof}
Let $0^d$ be the zero vector in dimension $d$ and, for each $1\leq i\leq d$, let $e_i$ be the $i$th standard basis vector. For each $d\geq1$, define
\[S_1^d:=\{0^d\}.\]
For each even $d\geq2$ define
\[S_2^d:=S_1^d\cup\{e_{2i-1}+e_{2i}: 1\leq i\leq d/2\}\]
and, for odd $d\geq3$, define
\[S_2^d:=S_1^d\cup\{e_{2i-1}+e_{2i}: 1\leq i\leq \lfloor d/2\rfloor\}\cup\{e_{d-1}+e_d\}.\]
Clearly, $|S_1^d|=1,|S_2^d|=\left\lceil\frac{d}{2}\right\rceil+1$ and $S_1^d\subseteq S_2^d$ for all $d\geq2$. The fact that $S_1^d$ percolates with respect to the $1$-neighbour bootstrap process is trivial and it is easy to verify that $S_2^d$ percolates with respect to the $2$-neighbour bootstrap process; in fact, $S_2^d$ coincides with the set $A_0$ in Construction~\ref{const:Steiner} specialized to $r=2$.

We prove, by induction on $d$, that, for every $d\geq3$, there exists a set $S_3^d\subseteq V(Q_d)$ of cardinality $\left\lceil\frac{d(d+3)}{6}\right\rceil+1$ that percolates under the $3$-neighbour bootstrap process in $Q_d$ and contains $S_2^d$. Our base cases are $d=3,4,5,6,7,8$.

\begin{case}
$d=3$.
\end{case}

Let $S_3^3$ be the set of all vertices of $Q_3$ whose coordinate sum is even. Clearly $|S_3^3|=4=\frac{3(3+3)}{6}+1$ and $S_2^3\subseteq S_3^3$. It is easy to see that $S_3^3$ percolates with respect to the $3$-neighbour bootstrap process in $Q_3$ (in one step). 

\begin{case}
$d=4$.
\end{case}

Define
\[S_3^4:=S_2^4\cup\{(0,1,1,0),
(1,0,0,1),
(1,1,1,1)\}.\]
Clearly, $|S_3^4|=6=\left\lceil\frac{4(4+3)}{6}\right\rceil+1$. The vertices 
\[(1,0,0,0),\ (0,1,0,0),\ (0,0,1,0),\ (0,0,0,1),\ (1,1,1,0),\ (1,1,0,1),\ (1,0,1,1),\ (0,1,1,1)\]
all have three neighbours in $S_3^4$ and so become infected after one round of the $3$-neighbour bootstrap process starting with $S_3^4$. At this point, all remaining uninfected vertices have three infected neighbours and so $S_3^4$ percolates. We also verify this by computer; see this \href{https://colab.research.google.com/drive/1_owaZNVjQTtzwKwdJ38Z51VGLKjKb-Fd?usp=sharing}{Google Colab Notebook} and the ancillary file \texttt{FourNeighbourBootstrapVerification.ipynb} uploaded with the arxiv version of this paper. 

\begin{case}
$d=5$.
\end{case}

Define
\[S_3^5:=S_2^5\cup\{(1, 0, 1, 1, 1), (0, 1, 0, 0, 1), (1, 1, 0, 1, 1), (0, 1, 1, 0, 0)\}.\]
Clearly, $|S_3^5|=8=\left\lceil\frac{5(5+3)}{6}\right\rceil+1$. The fact that $S_3^5$ percolates with respect to the $3$-neighbour process in $Q_5$ is verified by the Python program in this \href{https://colab.research.google.com/drive/1_owaZNVjQTtzwKwdJ38Z51VGLKjKb-Fd?usp=sharing}{Google Colab Notebook} and the ancillary file \texttt{FourNeighbourBootstrapVerification.ipynb} uploaded with the arxiv version of this paper. 

\begin{case}
$d=6$.
\end{case}

Define
\begin{align*}
S_3^6:=S_2^6\cup \{&(1, 1, 0, 0, 1, 1),  (1, 1, 0, 1, 0, 1), (0, 1, 1, 1, 0, 1), (1, 0, 1, 0, 1, 1),\\& (1, 0, 1, 0, 0, 0), (0, 0, 0, 1, 1, 0)\}.
\end{align*}
Clearly, $|S_3^6|=10 = \frac{6(6+3)}{6}+1$. The fact that $S_3^6$ percolates with respect to the $3$-neighbour process in $Q_6$ is verified by the Python program in this \href{https://colab.research.google.com/drive/1_owaZNVjQTtzwKwdJ38Z51VGLKjKb-Fd?usp=sharing}{Google Colab Notebook} and the ancillary file \texttt{FourNeighbourBootstrapVerification.ipynb} uploaded with the arxiv version of this paper.

\begin{case}
$d=7$.
\end{case}

Define 
\begin{align*}
S_3^7:=S_2^7\cup\{
&(1, 0, 0, 1, 1, 1, 1), (0, 1, 0, 0, 1, 0, 0), (1, 0, 0, 0, 0, 0, 1), (1, 0, 1, 0, 1, 0, 0),\\& (1, 1, 1, 1, 1, 0, 1), (1, 0, 0, 1, 0, 0, 0), (1, 0, 1, 1, 1, 1, 0),  (1, 1, 0, 1, 1, 1, 0)\}
\end{align*}
Clearly, $|S_3^7|=13=\left\lceil\frac{7(7+3)}{6}\right\rceil+1$. The fact that $S_3^7$ percolates with respect to the $3$-neighbour process in $Q_7$ is verified by the Python program in this \href{https://colab.research.google.com/drive/1_owaZNVjQTtzwKwdJ38Z51VGLKjKb-Fd?usp=sharing}{Google Colab Notebook} and the ancillary file \texttt{FourNeighbourBootstrapVerification.ipynb} uploaded with the arxiv version of this paper. 

\begin{case}
$d=8$.
\end{case}

Define 
\begin{align*}
S_3^8:=S_2^8\cup\{
&(0, 1, 0, 1, 1, 0, 0, 1), (1, 0, 1, 0, 1, 1, 0, 0), (0, 0, 1, 0, 1, 0, 0, 1), (0, 1, 0, 0, 1, 1, 1, 1),\\& (1, 1, 0, 1, 0, 1, 0, 0), (1, 0, 0, 1, 0, 1, 0, 0), (1, 0, 1, 1, 0, 1, 1, 0), (1, 1, 1, 0, 1, 1, 1, 1),\\& (0, 0, 1, 0, 0, 1, 1, 1), (0, 1, 0, 0, 0, 0, 0, 1), (0, 0, 1, 0, 0, 1, 0, 0)\}
\end{align*}
Clearly, $|S_3^8|=16=\left\lceil\frac{8(8+3)}{6}\right\rceil+1$. The fact that $S_3^8$ percolates with respect to the $3$-neighbour process in $Q_8$ is verified by the Python program in this \href{https://colab.research.google.com/drive/1_owaZNVjQTtzwKwdJ38Z51VGLKjKb-Fd?usp=sharing}{Google Colab Notebook} and the ancillary file \texttt{FourNeighbourBootstrapVerification.ipynb} uploaded with the arxiv version of this paper.

\begin{case}
$d\geq9$ and $d$ is odd.
\end{case}

We construct $S_3^d$ via an application of Lemma~\ref{lem:meta} with $k=3$. This construction comes from~\cite[Lemma~5.1]{MorrisonNoel18}. Consider the labeling $\ell_0^3:V(Q_3)\to\{0,1,2,3\}$ such that
\[\ell_0^3(0,0,0)=3,\]
\[\ell_0^3(1,1,0)=\ell_0^3(0,1,1) = 2,\]
\[\ell_0^3(1,0,1)=1,\]
and all other vertices are labeled $0$. It is not hard to verify $\ell_0^3$ percolates with respect to the $3$-meta bootstrap process in $Q_3$; see this \href{https://colab.research.google.com/drive/1_owaZNVjQTtzwKwdJ38Z51VGLKjKb-Fd?usp=sharing}{Google Colab Notebook} and the ancillary file \texttt{FourNeighbourBootstrapVerification.ipynb} uploaded with the arxiv version of this paper.

By the inductive hypothesis, there exists a percolating set $S_3^{d-3}$ for the $3$-neighbour bootstrap process in $Q_{d-3}$ of cardinality $\left\lceil\frac{(d-3)d}{6}\right\rceil+1$ which contains $S_2^{d-3}$ as a subset. Let $A_0$ be the subset of $V(Q_d)$ obtained by applying Lemma~\ref{lem:meta} to $\ell_0:=\ell_0^3$ and the sets $S_1:=S_1^{d-3},S_2:=S_2^{d-3}$ and $S_3:=S_3^{d-3}$ and define $S_3^d:=A_0$. The fact that $S_3^d$ percolates under the $3$-neighbour bootstrap process in $Q_d$ follows from Lemma~\ref{lem:meta}. Also, by construction and the fact that $d-3$ is even,
\begin{align*}|S_3^d|=|S_3^{d-3}|+2|S_2^{d-3}| + |S_1^{d-3}|& = \left(\left\lceil\frac{(d-3)d}{6}\right\rceil+1\right) + 2\left(\frac{d-3}{2}+1\right) + 1\\
&=\left\lceil\frac{d(d+3)}{6}\right\rceil+1.\end{align*}
Now, we show that $S_2^d\subseteq S_3^d$. Since $\ell_0^3(0,0,0)=3$ and $S_2^{d-3}\subseteq S_3^{d-3}$ by induction, every vertex obtained from a vertex of $S_2^{d-3}$ by prefixing it with three zeros is in $S_3^d$. Also, since $\ell_0(1,1,0)$ and $\ell_0(0,1,1)$ are non-zero and $S_1^{d-3}\subseteq S_2^{d-3}$, the two vertices obtained from $0^{d-3}$ by prefixing with $(1,1,0)$ or $(0,1,1)$ are contained in $S_3^d$. Thus, $S_2^d\subseteq S_3^d$ and we are done. 

\begin{case}
$d\geq9$ and $d$ is even.
\end{case}

We construct $S_3^d$ via an application of Lemma~\ref{lem:meta} with $k=6$. Consider the labeling $\ell_0^6:V(Q_6)\to\{0,1,2,3\}$ such that
\[\ell_0^6(1,1,0,0,0,0)=3,\]
\[\ell_0^6(0,0,0,0,0,0)=\ell_0^6(0,0,0,0,1,1) = \ell_0^6(0,0,0,1,0,1)=\ell_0^6(1,0,1,0,0,0)=2,\]
\begin{align*}&\ell_0^6(0,0,1,1,0,0)=\ell_0^6(0,1,0,1,1,0)=\ell_0^6(0,1,1,1,1,1)=\ell_0^6(1,0,1,1,0,1)\\
&=\ell_0^6(1,1,0,1,1,1)=1,\end{align*}
and all other vertices are labeled $0$. The fact that $\ell_0^6$ percolates with respect to the $3$-meta bootstrap process in $Q_6$ is verified by the Python program in this \href{https://colab.research.google.com/drive/1_owaZNVjQTtzwKwdJ38Z51VGLKjKb-Fd?usp=sharing}{Google Colab Notebook} and the ancillary file \texttt{FourNeighbourBootstrapVerification.ipynb} uploaded with the arxiv version of this paper. 

By the inductive hypothesis, there exists a percolating set $S_3^{d-6}$ for the $3$-neighbour bootstrap process in $Q_{d-6}$ of cardinality $\frac{(d-6)(d-3)}{6}+1$ which contains $S_2^{d-6}$ as a subset. Let $A_0$ be the subset of $V(Q_d)$ obtained by applying Lemma~\ref{lem:meta} to $\ell_0:=\ell_0^6$ and the sets $S_1:=S_1^{d-6},S_2:=S_2^{d-6}$ and $S_3:=S_3^{d-6}$ and define $S_3^d:=A_0$. The fact that $S_3^d$ percolates under the $3$-neighbour bootstrap process in $Q_d$ follows from Lemma~\ref{lem:meta}. Also, by construction,
\begin{align*}|S_3^d|=|S_3^{d-6}|+4|S_2^{d-6}| + 5|S_1^{d-6}|& = \left(\left\lceil\frac{(d-6)(d-3)}{6}\right\rceil+1\right) + 4\left(\frac{d-6}{2}+1\right) + 5\\
&=\left\lceil\frac{d(d+3)}{6}\right\rceil+1.\end{align*}
Now, we show that $S_2^d\subseteq S_3^d$. Since $\ell_0^6(0,0,0,0,0,0)=2$, we have that every vertex obtained from a vertex of $S_2^{d-6}$ by prefixing it with six zeros is in $S_3^d$. Also, since $\ell_0^6(1,1,0,0,0,0),\ell_0^6(0,0,1,1,0,0)$ and $\ell_0^6(0,0,0,0,1,1)$ are all non-zero and $S_1^{d-6}\subseteq S_2^{d-6}\subseteq S_3^{d-6}$, the three vertices obtained from $0^{d-6}$ by prefixing with $(1,1,0,0,0,0),(0,0,1,1,0,0)$ or $(0,0,0,0,1,1)$ are contained in $S_3^d$. Thus, $S_2^d\subseteq S_3^d$ and we are done. 
\end{proof}

\section{4-Neighbour Seed Constructions}
\label{sec:constructionsSeed4}

In this section, we provide explicit constructions of percolating sets for the $4$-neighbour bootstrap process in $Q_d$ for $4\leq d\leq 15$  which, together with the constructions in the previous section, form the atomic ``building blocks'' that are fed into Lemma~\ref{lem:meta} to prove our main theorems. 

\begin{lem}
\label{lem:Q4}
There exists a percolating set $S_4^4$ of cardinality $8$ for the $4$-neighbour bootstrap process in $Q_4$. 
\end{lem}

\begin{proof}
Take $S_4^4$ to be the set of all vertices of $Q_4$ whose coordinate sum is even. Then $S_4^4$ is easily seen to have cardinality 8 and to percolate under the $4$-neighbour bootstrap process in $Q_4$ (in only one step). 
\end{proof}

Next, we consider $d=5$. Note that the size of the percolating set that we find in this case is one larger than the lower bound in Theorem~\ref{th:MN}, which is consistent with the exhaustive computer search discussed in~\cite[Section~6]{MorrisonNoel18}.

\begin{lem}
\label{lem:Q5}
There exists a percolating set $S_4^5$ of cardinality $14$ for the $4$-neighbour bootstrap process in $Q_5$. 
\end{lem}

\begin{proof}
Define
\begin{align*}
S_4^5:=
\{
&(0,0,0,0,1),
(0,0,1,0,0),
(0,0,1,1,1),
(0,1,0,1,0),
(0,1,0,1,1),
\\&(0,1,1,0,1),
(0,1,1,1,0),
(1,0,0,0,0),
(1,0,0,1,1),
(1,0,1,0,1),
\\&(1,0,1,1,0),
(1,1,0,0,1),
(1,1,0,1,0),
(1,1,1,0,0)
\}.
\end{align*}
Clearly, $|S_4^5|=14$. For a verification program showing that $S_4^5$ percolates under the $4$-neighbour bootstrap process in $Q_5$, see this \href{https://colab.research.google.com/drive/1_owaZNVjQTtzwKwdJ38Z51VGLKjKb-Fd?usp=sharing}{Google Colab Notebook} and the ancillary file \texttt{FourNeighbourBootstrapVerification.ipynb} uploaded with the arxiv version of this paper. 
\end{proof}

Next, for $d=6$, we present a construction which matches the lower bound in Theorem~\ref{th:MN}.

\begin{lem}
\label{lem:Q6}
There exists a percolating set $S_4^6$ of cardinality $18$ for the $4$-neighbour bootstrap process in $Q_6$. 
\end{lem}

\begin{proof}
Define
\begin{align*}
S_4^6:=
\{
&(1,0,0,0,0,1),
(1,0,0,0,1,0),
(0,0,0,0,1,1),
(0,0,0,1,0,0),
(1,0,0,1,1,1),
\\&(0,0,1,0,1,0),
(1,0,1,1,0,1),
(1,0,1,1,1,0),
(1,1,0,0,0,0),
(0,1,0,0,1,0),
\\&(1,1,0,1,0,1),
(1,1,0,1,1,0),
(0,1,0,1,1,1),
(0,1,1,0,0,0),
(1,1,1,0,1,1),
\\&(1,1,1,1,0,0),
(0,1,1,1,0,1),
(0,1,1,1,1,0)
\}.
\end{align*}
Clearly, $|S_4^6|=18$. For a verification program showing that $S_4^6$ percolates under the $4$-neighbour bootstrap process in $Q_6$, see this \href{https://colab.research.google.com/drive/1_owaZNVjQTtzwKwdJ38Z51VGLKjKb-Fd?usp=sharing}{Google Colab Notebook} and the ancillary file \texttt{FourNeighbourBootstrapVerification.ipynb} uploaded with the arxiv version of this paper. 
\end{proof}

The next construction is for $d=7$. This matches the construction found by B\'erczi and Wagner~\cite{BercziWagner24} and the lower bound in Theorem~\ref{th:MN}.

\begin{lem}
\label{lem:Q7}
There exists a percolating set $S_4^7$ of cardinality $26$ for the $4$-neighbour bootstrap process in $Q_7$. 
\end{lem}

\begin{proof}
Define
\begin{align*}
S_4^7:=
\{
&
(0,0,0,0,0,1,0),
(0,0,0,0,1,1,1),
(0,0,0,1,0,1,1),
(0,0,0,1,1,0,1),
(0,0,1,0,1,0,1),
\\&(0,0,1,1,0,0,1),
(0,0,1,1,0,1,0),
(0,1,0,0,1,1,0),
(0,1,0,1,0,1,0),
(0,1,0,1,1,0,0),
\\&(0,1,1,0,1,1,1),
(0,1,1,1,0,0,0),
(0,1,1,1,0,1,1),
(1,0,0,0,0,1,1),
(1,0,0,1,1,1,1),
\\&(1,0,1,0,0,0,1),
(1,0,1,0,0,1,0),
(1,0,1,0,1,0,0),
(1,0,1,1,1,0,1),
(1,1,0,0,0,1,0),
\\&(1,1,0,0,1,0,0),
(1,1,0,0,1,1,1),
(1,1,0,1,0,0,1),
(1,1,0,1,1,1,0),
(1,1,1,0,0,1,1),
\\&(1,1,1,1,1,0,0)
\}.
\end{align*}
Clearly, $|S_4^7|=26$. For a verification program showing that $S_4^7$ percolates under the $4$-neighbour bootstrap process in $Q_7$, see this \href{https://colab.research.google.com/drive/1_owaZNVjQTtzwKwdJ38Z51VGLKjKb-Fd?usp=sharing}{Google Colab Notebook} and the ancillary file \texttt{FourNeighbourBootstrapVerification.ipynb} uploaded with the arxiv version of this paper. 
\end{proof}

The next construction is for $d=8$. This matches the construction found by B\'erczi and Wagner~\cite{BercziWagner24} and the lower bound in Theorem~\ref{th:MN}.

\begin{lem}
\label{lem:Q8}
There exists a percolating set $S_4^8$ of cardinality $35$ for the $4$-neighbour bootstrap process in $Q_8$. 
\end{lem}

\begin{proof}
Define
\begin{align*}
S_4^8:=
\{
&(1,0,0,0,0,0,1,1),
(0,0,0,0,0,1,0,0),
(1,0,0,0,0,1,1,0),
(1,0,0,0,1,0,0,1),
\\&(0,0,0,0,1,1,1,0),
(0,0,0,1,0,0,1,1),
(1,0,0,1,0,1,0,0),
(0,0,0,1,0,1,1,0),
\\&(0,0,0,1,1,0,0,0),
(0,0,0,1,1,1,1,1),
(1,0,1,0,0,0,0,1),
(1,0,1,0,0,1,0,0),
\\&(1,0,1,0,0,1,1,1),
(1,0,1,0,1,1,1,0),
(0,0,1,0,1,1,1,1),
(1,0,1,1,0,0,0,0),
\\&(0,0,1,1,0,1,1,1),
(0,0,1,1,1,0,0,1),
(0,0,1,1,1,1,1,0),
(0,1,0,0,0,0,0,0),
\\&(1,1,0,0,0,0,0,1),
(1,1,0,0,0,1,1,1),
(0,1,0,0,1,0,0,1),
(1,1,0,0,1,0,1,1),
\\&(1,1,0,0,1,1,1,0),
(0,1,0,0,1,1,1,1),
(0,1,0,1,0,0,1,0),
(1,1,0,1,0,1,1,0),
\\&(0,1,0,1,1,0,1,1),
(1,1,0,1,1,1,0,0),
(0,1,1,0,0,0,0,1),
(0,1,1,0,0,1,0,0),
\\&(1,1,1,0,0,1,0,1),
(0,1,1,0,1,1,0,1),
(1,1,1,1,1,0,1,1)
\}.
\end{align*}
Clearly, $|S_4^8|=35$. For a verification program showing that $S_4^8$ percolates under the $4$-neighbour bootstrap process in $Q_8$, see this \href{https://colab.research.google.com/drive/1_owaZNVjQTtzwKwdJ38Z51VGLKjKb-Fd?usp=sharing}{Google Colab Notebook} and the ancillary file \texttt{FourNeighbourBootstrapVerification.ipynb} uploaded with the arxiv version of this paper. 
\end{proof}

The next construction is for $d=9$. This improves the construction found by B\'erczi and Wagner~\cite{BercziWagner24} by one element and matches the lower bound in Theorem~\ref{th:MN}.

\begin{lem}
\label{lem:Q9}
There exists a percolating set $S_4^9$ of cardinality $47$ for the $4$-neighbour bootstrap process in $Q_9$. 
\end{lem}

\begin{proof}
Define
\begin{align*}
S_4^9:=
\{&(1,1,0,0,0,0,0,0,0),
(0,0,0,0,0,0,0,1,1),
(1,0,0,0,0,1,0,0,0),
(0,1,0,0,0,1,0,1,1),
\\&(0,1,0,0,1,0,0,0,0),
(1,0,0,0,1,0,0,1,1),
(0,0,0,0,1,0,1,0,0),
(1,1,0,0,1,0,1,0,0),
\\&(0,1,0,0,1,1,0,1,0),
(1,0,0,0,1,1,1,0,0),
(0,1,0,1,0,0,0,0,1),
(0,0,0,1,0,0,0,1,0),
\\&(0,0,0,1,0,0,1,0,1),
(1,1,0,1,0,1,0,0,0),
(0,0,0,1,0,1,0,0,0),
(0,1,0,1,0,1,0,1,0),
\\&(1,0,0,1,0,1,0,1,1),
(1,0,0,1,0,1,1,0,0),
(0,1,0,1,0,1,1,0,1),
(1,1,0,1,1,0,0,0,0),
\\&(1,0,0,1,1,0,0,1,0),
(0,1,0,1,1,0,0,1,1),
(0,1,0,1,1,0,1,0,0),
(0,0,0,1,1,0,1,1,1),
\\&(0,1,0,1,1,1,0,0,0),
(1,0,0,1,1,1,0,0,1),
(0,0,0,1,1,1,0,1,1),
(0,0,0,1,1,1,1,1,0),
\\&(1,0,0,1,1,1,1,1,1),
(0,1,1,0,0,0,0,0,0),
(0,0,1,0,0,0,0,1,0),
(1,1,1,0,0,0,0,1,0),
\\&(1,1,1,0,0,1,0,0,0),
(0,0,1,0,0,1,0,0,1),
(1,0,1,0,0,1,1,1,1),
(0,0,1,0,1,0,0,0,0),
\\&(0,1,1,0,1,0,0,1,1),
(1,1,1,0,1,0,1,0,1),
(0,1,1,0,1,1,0,0,0),
(1,0,1,0,1,1,0,0,0),
\\&(0,0,1,1,0,0,0,0,0),
(1,0,1,1,0,1,0,0,0),
(0,0,1,1,0,1,0,1,0),
(0,0,1,1,0,1,1,0,1),
\\&(0,0,1,1,1,0,1,0,0),
(0,1,1,1,1,0,1,1,0),
(0,1,1,1,1,1,1,0,1)
\}.
\end{align*}
Clearly, $|S_4^9|=47$. For a verification program showing that $S_4^9$ percolates under the $4$-neighbour bootstrap process in $Q_9$, see this \href{https://colab.research.google.com/drive/1_owaZNVjQTtzwKwdJ38Z51VGLKjKb-Fd?usp=sharing}{Google Colab Notebook} and the ancillary file \texttt{FourNeighbourBootstrapVerification.ipynb} uploaded with the arxiv version of this paper. 
\end{proof}

The next construction is for $d=10$. This improves the construction found by B\'erczi and Wagner~\cite{BercziWagner24} by two elements and matches the lower bound in Theorem~\ref{th:MN}.

\begin{lem}
\label{lem:Q10}
There exists a percolating set $S_4^{10}$ of cardinality $61$ for the $4$-neighbour bootstrap process in $Q_{10}$. 
\end{lem}

\begin{proof}
Define
{\footnotesize
\begin{align*}
S_4^{10}:=
\{&
(0,1,1,0,0,0,0,1,1,0),
(1,1,0,1,1,1,1,1,0,1),
(0,0,1,0,0,0,1,1,1,0),
(0,0,0,0,0,0,1,1,1,1),
\\&(1,0,1,0,0,1,0,1,0,0),
(1,0,1,0,0,1,1,0,1,1),
(0,1,0,0,0,1,1,1,0,1),
(1,0,0,0,0,1,1,1,1,1),
\\&(1,1,0,0,1,0,0,0,0,1),
(1,0,1,0,1,0,0,0,0,1),
(1,1,1,0,1,0,0,0,1,1),
(1,0,0,0,1,0,0,1,0,1),
\\&(1,0,1,0,1,0,0,1,1,1),
(1,1,0,0,1,0,1,0,0,0),
(1,0,0,0,1,0,1,0,0,1),
(0,1,1,0,1,0,1,0,1,1),
\\&(0,1,1,0,1,0,1,1,1,0),
(1,1,1,0,1,0,1,1,1,1),
(1,1,1,0,1,1,0,1,1,1),
(0,0,1,0,1,1,1,1,0,1),
\\&(1,0,1,0,1,1,1,1,1,0),
(0,0,0,0,1,1,1,1,1,1),
(0,1,0,1,0,0,0,0,0,0),
(0,1,1,0,1,1,1,1,1,1),
\\&(1,1,1,1,0,0,0,0,0,0),
(1,1,1,1,0,0,0,1,1,0),
(1,1,0,1,0,0,0,1,1,1),
(0,1,1,1,0,0,1,0,0,0),
\\&(0,0,0,1,0,0,1,0,1,1),
(1,0,1,1,0,0,1,0,1,1),
(1,1,0,1,0,0,1,0,1,1),
(0,0,0,1,0,0,1,1,1,0),
\\&(1,1,0,1,0,0,1,1,1,0),
(0,1,0,1,0,1,0,1,0,1),
(0,1,0,1,0,1,0,1,1,0),
(1,0,0,1,0,1,0,1,1,1),
\\&(1,1,1,1,0,1,0,1,1,1),
(0,0,1,1,0,1,1,0,0,0),
(1,0,1,1,0,1,1,1,1,1),
(0,1,1,1,0,1,1,1,1,1),
\\&(1,1,1,1,1,0,0,0,0,1),
(1,0,1,1,1,0,0,1,0,1),
(0,0,0,1,1,0,0,1,1,0),
(1,1,1,1,1,0,1,0,0,0),
\\&(0,0,1,1,1,0,1,0,0,0),
(0,1,1,1,1,0,1,0,0,1),
(1,0,1,1,1,0,1,0,1,0),
(1,0,0,1,1,0,1,0,1,1),
\\&(1,1,1,1,1,0,1,1,0,1),
(1,1,1,1,1,0,1,1,1,0),
(1,0,1,1,1,0,1,1,1,1),
(0,1,0,1,1,0,1,0,1,1),
\\&(1,0,1,1,1,1,0,0,0,1),
(0,0,0,1,1,0,1,1,1,1),
(0,0,1,1,1,1,1,0,1,0),
(1,1,0,1,1,1,0,1,1,1),
\\&(0,0,1,1,1,1,1,0,0,1),
(1,0,0,1,1,1,1,0,1,0),
(0,1,1,1,1,1,1,0,1,1),
(1,0,1,1,1,1,1,1,0,1),
\\&(1,0,0,1,1,1,1,1,1,1)
\}
\end{align*}}
Clearly, $|S_4^{10}|=61$. For a verification program showing that $S_4^{10}$ percolates under the $4$-neighbour bootstrap process in $Q_{10}$, see this \href{https://colab.research.google.com/drive/1_owaZNVjQTtzwKwdJ38Z51VGLKjKb-Fd?usp=sharing}{Google Colab Notebook} and the ancillary file \texttt{FourNeighbourBootstrapVerification.ipynb} uploaded with the arxiv version of this paper. 
\end{proof}

The next construction is for $d=11$. This matches the construction found by B\'erczi and Wagner~\cite{BercziWagner24} by one element and matches the lower bound in Theorem~\ref{th:MN}.

\begin{lem}
\label{lem:Q11}
There exists a percolating set $S_4^{11}$ of cardinality $78$ for the $4$-neighbour bootstrap process in $Q_{11}$. 
\end{lem}

\begin{proof}
Define
{\scriptsize
\begin{align*}
S_4^{11}:=
\{&(1,1,0,0,0,0,0,0,1,0,0),
(1,1,0,0,0,0,0,1,0,0,1),
(1,1,0,0,0,0,0,1,0,1,0),
(1,0,0,0,0,0,0,1,1,0,0),
\\&(1,1,0,0,0,0,1,0,1,1,0),
(0,1,0,0,0,0,1,1,0,1,0),
(1,1,0,0,0,1,0,0,0,0,0),
(1,1,0,0,0,1,0,0,1,1,0),
\\&(1,0,0,0,0,1,0,1,0,0,0),
(1,1,0,0,0,1,0,1,0,1,1),
(1,1,0,0,0,1,0,1,1,0,0),
(1,1,0,0,0,1,1,0,0,1,0),
\\&(0,1,0,0,0,1,1,0,1,0,0),
(0,0,0,0,0,1,1,0,1,1,0),
(1,1,0,0,0,1,1,1,0,0,1),
(0,0,0,0,0,1,1,1,0,1,0),
\\&(0,0,0,0,0,1,1,1,0,0,1),
(1,0,0,0,1,1,0,1,0,1,0),
(0,0,0,0,1,1,0,1,0,1,1),
(0,1,0,0,1,1,0,1,0,1,0),
\\&(0,1,0,0,1,1,1,0,1,1,1),
(1,1,0,0,1,1,1,1,0,1,1),
(1,1,0,1,0,0,0,0,1,1,0),
(1,0,0,1,0,0,0,1,0,1,0),
\\&(1,1,0,1,0,0,0,1,0,1,1),
(0,0,0,1,0,0,0,1,1,1,0),
(0,1,0,1,0,0,0,1,1,1,1),
(0,1,0,1,0,0,1,0,0,1,1),
\\&(1,1,0,1,0,0,1,1,1,0,0),
(1,0,0,1,0,1,0,0,0,1,0),
(0,1,0,1,0,1,0,0,0,1,1),
(0,1,0,1,0,1,0,0,1,1,0),
\\&(1,1,0,1,0,1,0,1,0,0,0),
(0,0,0,1,0,1,0,1,0,1,0),
(1,1,0,1,0,1,1,0,0,0,0),
(0,0,0,1,0,1,1,0,0,1,0),
\\&(0,1,0,1,0,1,1,1,0,0,0),
(1,1,0,1,0,1,1,1,0,1,0),
(1,0,0,1,0,1,1,1,0,1,1),
(0,1,0,1,0,1,1,1,1,1,0),
\\&(1,1,0,1,0,1,1,1,1,1,1),
(1,1,0,1,1,0,0,0,0,0,0),
(0,0,0,1,1,0,0,0,0,1,0),
(1,1,0,1,1,0,0,1,0,0,1),
\\&(1,1,0,1,1,0,0,1,0,1,0),
(0,1,0,1,1,0,0,1,1,0,0),
(1,0,0,1,1,0,0,1,1,1,0),
(0,1,0,1,1,0,1,1,1,1,0),
\\&(0,0,0,1,1,1,0,0,0,0,0),
(0,1,0,1,1,1,0,1,0,1,1),
(0,1,0,1,1,1,0,1,1,1,0),
(0,1,0,1,1,1,1,0,0,1,0),
\\&(1,1,1,0,0,0,1,0,0,1,0),
(1,0,1,0,0,0,1,0,1,1,0),
(1,0,1,0,0,0,1,1,0,0,1),
(1,0,1,0,0,0,1,1,1,0,0),
\\&(0,1,1,0,0,1,0,0,1,1,0),
(1,0,1,0,0,1,0,0,1,1,0),
(1,0,1,0,0,1,0,1,0,0,1),
(0,0,1,0,0,1,1,0,0,1,0),
\\&(0,0,1,0,0,1,1,0,1,0,0),
(1,1,1,0,0,1,1,1,0,0,0),
(0,1,1,0,0,1,1,1,0,1,0),
(1,1,1,0,0,1,1,1,0,1,1),
\\&(1,0,1,0,1,0,1,0,0,0,1),
(1,0,1,0,1,1,0,0,0,1,0),
(0,0,1,0,1,1,0,0,1,0,1),
(1,1,1,0,1,1,0,1,0,1,1),
\\&(1,1,1,0,1,1,1,0,0,1,0),
(1,1,1,0,1,1,1,1,1,1,1),
(0,1,1,1,0,0,0,0,1,1,1),
(0,1,1,1,0,0,1,0,0,0,0),
\\&(1,0,1,1,0,1,1,0,0,1,0),
(0,1,1,1,0,1,1,0,1,0,0),
(0,0,1,1,0,1,1,0,1,1,0),
(1,0,1,1,0,1,1,1,1,1,1),
\\&(1,0,1,1,1,0,0,0,1,0,0),
(0,1,1,1,1,1,0,0,0,1,0)
\}.
\end{align*}}
Clearly, $|S_4^{11}|=78$. For a verification program showing that $S_4^{11}$ percolates under the $4$-neighbour bootstrap process in $Q_{11}$, see this \href{https://colab.research.google.com/drive/1_owaZNVjQTtzwKwdJ38Z51VGLKjKb-Fd?usp=sharing}{Google Colab Notebook} and the ancillary file \texttt{FourNeighbourBootstrapVerification.ipynb} uploaded with the arxiv version of this paper. 
\end{proof}

The next construction is for $d=12$. This improves the construction found by B\'erczi and Wagner~\cite{BercziWagner24} by one element and matches the lower bound in Theorem~\ref{th:MN}.

\begin{lem}
\label{lem:Q12}
There exists a percolating set $S_4^{12}$ of cardinality $98$ for the $4$-neighbour bootstrap process in $Q_{12}$. 
\end{lem}

\begin{proof}
Define
{\scriptsize
\begin{align*}
S_4^{12}:=
\{
&(0,1,1,0,0,0,0,0,0,0,1,0),
(0,0,1,0,0,0,0,0,0,1,1,0),
(0,1,0,0,0,0,0,0,1,1,1,1),
(1,1,1,0,0,0,1,0,0,0,0,1),
\\&(1,0,0,0,0,1,0,0,0,0,0,0),
(1,1,1,0,0,1,0,0,0,0,0,1),
(0,0,1,0,0,1,0,0,0,0,0,1),
(0,0,0,0,0,1,0,0,0,1,1,0),
\\&(1,0,1,0,0,1,0,0,1,0,0,1),
(0,0,1,0,0,1,0,0,1,0,1,1),
(1,1,1,0,0,1,0,0,1,0,1,1),
(0,0,0,0,0,1,0,0,1,1,1,1),
\\&(0,1,1,0,0,1,0,0,1,1,1,1),
(0,1,1,0,0,1,0,1,0,0,1,0),
(1,1,1,0,0,1,0,1,0,0,1,1),
(1,1,0,0,0,1,0,1,0,0,1,0),
\\&(0,0,0,0,0,1,0,1,0,1,0,1),
(0,1,0,0,0,1,0,1,0,1,1,0),
(0,0,1,0,0,1,0,1,0,1,1,0),
(0,1,1,0,0,1,0,1,1,1,0,0),
\\&(0,1,1,0,0,1,1,0,1,1,0,1),
(1,0,0,0,0,1,1,0,1,1,1,1),
(0,0,0,0,1,0,0,0,0,1,1,0),
(0,0,1,0,1,0,0,0,1,1,1,0),
\\&(0,0,1,0,1,0,0,1,0,1,1,0),
(0,0,0,0,1,0,1,0,1,1,1,0),
(1,0,0,0,1,0,1,1,1,0,1,0),
(0,0,0,0,1,1,0,0,0,1,1,1),
\\&(0,1,1,0,1,1,0,0,1,0,1,1),
(0,0,0,0,1,1,0,0,1,1,1,0),
(1,0,0,0,1,1,0,0,1,1,1,1),
(0,0,1,0,1,1,1,0,0,0,0,1),
\\&(0,0,1,0,1,1,1,0,0,1,1,1),
(1,0,0,0,1,1,1,0,0,1,1,1),
(0,0,1,0,1,1,1,0,1,1,0,1),
(1,0,1,1,0,0,0,0,0,1,1,0),
\\&(1,1,1,1,0,0,0,0,0,1,1,1),
(0,1,1,1,0,0,0,0,0,1,1,0),
(0,0,1,1,0,0,0,1,0,1,1,0),
(1,1,1,1,0,0,0,1,0,1,1,0),
\\&(1,0,0,1,0,0,1,0,0,1,1,0),
(1,0,1,1,0,0,1,0,0,1,1,1),
(0,0,1,1,0,1,0,0,0,0,1,1),
(1,1,0,1,0,1,0,0,0,1,0,0),
\\&(1,1,1,1,0,1,0,0,0,0,1,1),
(0,1,0,1,0,1,0,0,0,1,1,0),
(0,1,1,1,0,1,0,0,0,1,1,1),
(1,1,1,1,0,1,0,0,0,1,0,1),
\\&(0,0,0,1,0,1,0,0,0,1,1,1),
(0,0,0,1,0,1,0,0,1,1,1,0),
(1,1,1,1,0,1,0,1,0,0,1,0),
(1,0,0,1,0,1,0,1,0,1,0,0),
\\&(1,1,0,1,0,1,0,1,0,1,1,0),
(0,1,1,1,0,1,0,1,0,1,1,0),
(1,0,1,1,0,1,0,1,1,1,0,0),
(1,1,1,1,0,1,1,0,0,0,0,1),
\\&(0,1,1,1,0,1,1,0,0,0,1,1),
(1,0,1,1,0,1,1,0,0,1,0,1),
(1,0,0,1,0,1,1,0,0,1,1,1),
(1,0,0,1,0,1,1,1,1,1,1,1),
\\&(0,0,1,1,1,0,0,0,0,1,0,1),
(0,1,0,1,1,0,0,0,0,1,0,1),
(1,0,1,1,1,0,0,0,1,1,0,1),
(1,1,1,1,1,0,0,1,0,0,1,0),
\\&(0,1,0,1,1,0,0,1,0,0,1,0),
(1,0,1,1,1,0,0,1,0,0,1,1),
(1,0,1,1,1,0,0,1,0,1,1,0),
(1,1,0,1,1,0,0,1,0,1,1,0),
\\&(1,0,1,1,1,0,0,1,1,0,0,1),
(1,0,1,1,1,0,0,1,1,0,1,0),
(1,0,1,1,1,0,0,1,1,1,1,1),
(1,1,1,1,1,0,1,0,0,0,0,0),
\\&(0,0,0,1,1,0,1,0,0,1,1,0),
(1,0,1,1,1,0,1,1,0,1,1,1),
(0,0,1,1,1,0,1,1,1,1,0,0),
(1,0,0,1,1,1,0,0,0,0,1,0),
\\&(0,0,0,1,1,1,0,0,0,0,1,1),
(0,1,0,1,1,1,0,0,0,0,1,0),
(0,1,1,1,1,1,0,0,0,1,0,1),
(1,0,1,1,1,1,0,0,1,0,0,1),
\\&(1,0,0,1,1,1,0,0,1,1,0,1),
(0,1,1,1,1,1,0,1,0,0,1,0),
(1,0,1,1,1,1,0,1,0,1,0,0),
(1,0,0,1,1,1,0,1,0,1,0,1),
\\&(1,0,0,1,1,1,0,1,0,1,1,0),
(0,0,1,1,1,1,0,1,0,1,1,0),
(1,0,1,1,1,1,0,1,1,1,0,1),
(1,0,0,1,1,1,0,1,1,1,1,1),
\\&(0,0,1,1,1,1,1,0,0,0,1,1),
(0,0,1,1,1,1,1,0,0,1,0,1),
(1,0,0,1,1,1,1,0,0,1,1,0),
(0,0,1,1,1,1,1,0,0,1,1,0),
\\&(0,0,0,1,1,1,1,0,1,1,1,0),
(0,0,1,1,1,1,1,1,0,0,0,1),
(1,1,0,1,1,1,1,1,0,1,0,0),
(1,0,0,1,1,1,1,1,0,1,1,1),
\\&(1,1,1,1,1,1,1,1,1,1,0,1),
(1,0,0,1,1,1,1,1,1,1,1,0)
\}.
\end{align*}}
Clearly, $|S_4^{12}|=98$. For a verification program showing that $S_4^{12}$ percolates under the $4$-neighbour bootstrap process in $Q_{12}$, see this \href{https://colab.research.google.com/drive/1_owaZNVjQTtzwKwdJ38Z51VGLKjKb-Fd?usp=sharing}{Google Colab Notebook} and the ancillary file \texttt{FourNeighbourBootstrapVerification.ipynb} uploaded with the arxiv version of this paper. 
\end{proof}

The next construction is for $d=13$. This matches the construction found by B\'erczi and Wagner~\cite{BercziWagner24} and the lower bound in Theorem~\ref{th:MN}.

\begin{lem}
\label{lem:Q13}
There exists a percolating set $S_4^{13}$ of cardinality $122$ for the $4$-neighbour bootstrap process in $Q_{13}$. 
\end{lem}

\begin{proof}
Define
{\scriptsize
\begin{align*}
S_4^{13}:=
\{
&(0,0,1,0,0,0,0,0,0,1,1,1,0),
(0,0,1,0,0,0,0,1,0,0,0,0,0),
(1,0,1,0,0,0,0,1,0,0,0,1,1),
(0,0,1,0,0,0,0,1,0,0,1,0,1),
\\&(0,0,1,0,0,0,0,1,0,1,1,0,0),
(1,1,1,1,0,0,0,1,1,1,0,1,1),
(1,0,0,0,0,0,0,1,1,1,0,1,1),
(1,1,0,1,0,0,1,0,0,0,0,0,1),
\\&(1,1,0,1,0,0,1,0,1,0,0,0,0),
(1,1,1,1,0,0,1,0,1,0,0,0,1),
(1,0,1,1,0,0,1,0,1,0,0,1,1),
(1,0,1,1,0,0,1,0,1,1,0,0,1),
\\&(1,0,0,1,0,0,1,0,1,1,0,1,1),
(1,1,1,1,0,0,1,0,1,1,0,1,1),
(1,1,1,1,0,0,1,1,0,0,0,0,0),
(1,0,1,0,0,0,1,1,0,0,0,0,1),
\\&(1,0,0,1,0,0,1,1,0,0,0,0,1),
(0,0,0,0,0,0,1,1,0,0,0,0,0),
(1,0,1,0,0,0,1,1,0,0,0,1,0),
(0,0,1,0,0,0,1,1,0,1,0,0,0),
\\&(0,1,1,0,0,0,1,1,0,1,0,0,1),
(1,1,0,1,0,0,1,1,0,1,0,1,0),
(0,0,1,1,0,0,1,1,0,1,1,0,0),
(1,0,0,1,0,0,1,1,1,0,0,1,1),
\\&(0,0,1,1,0,0,1,1,1,0,0,1,1),
(0,0,0,0,0,1,0,0,0,1,0,1,1),
(0,1,1,0,0,1,0,0,1,0,0,0,0),
(0,1,1,1,0,1,0,0,1,0,0,0,1),
\\&(0,0,1,1,0,1,0,0,1,0,0,1,0),
(1,0,1,0,0,1,0,0,1,0,0,0,1),
(0,0,0,0,0,1,0,1,0,0,0,0,0),
(0,0,1,0,0,1,0,1,0,0,0,0,1),
\\&(0,0,1,1,0,1,0,1,0,0,0,1,0),
(0,0,1,1,0,1,0,1,0,0,0,1,1),
(0,0,1,1,0,1,0,1,0,0,1,0,0),
(0,1,1,0,0,1,0,1,0,0,1,1,0),
\\&(0,0,1,0,0,1,0,1,0,1,0,0,0),
(0,0,0,1,0,1,0,1,0,1,0,1,1),
(0,0,1,0,0,1,0,1,0,1,0,1,1),
(0,0,0,0,0,1,0,1,0,1,1,1,0),
\\&(1,1,1,0,0,1,0,1,1,0,0,0,0),
(1,0,0,1,0,1,0,1,1,0,0,0,1),
(1,0,1,1,0,1,0,1,1,0,0,1,1),
(0,0,0,1,0,1,0,1,1,0,0,1,1),
\\&(0,0,1,1,0,1,0,1,1,0,1,0,1),
(0,0,1,1,0,1,0,1,1,0,1,1,0),
(1,0,1,0,0,1,0,1,1,0,1,0,1),
(1,0,0,1,0,1,0,1,1,1,0,0,0),
\\&(0,0,0,0,0,1,0,1,1,1,0,0,1),
(0,0,0,1,0,1,0,1,1,1,0,1,0),
(1,1,0,0,0,1,0,1,1,1,0,1,1),
(1,0,1,0,0,1,0,1,1,1,1,0,0),
\\&(1,0,0,0,0,1,0,1,1,1,1,0,1),
(0,1,0,1,0,1,0,1,1,1,0,1,1),
(0,0,1,1,0,1,0,1,1,1,0,0,1),
(0,1,1,0,0,1,0,1,1,0,1,1,1),
\\&(1,1,1,1,0,1,1,0,0,0,0,0,1),
(1,0,1,1,0,1,1,0,0,1,0,0,1),
(0,1,1,0,0,1,1,0,0,1,0,0,1),
(1,1,1,0,0,1,1,0,0,1,1,0,0),
\\&(0,1,1,1,0,1,1,0,1,0,0,0,0),
(0,0,1,1,0,1,1,0,1,0,0,1,1),
(1,0,0,1,0,1,1,0,1,0,0,1,1),
(1,0,1,0,0,1,1,0,1,0,1,0,1),
\\&(0,1,1,0,0,1,1,0,1,0,1,1,0),
(0,0,1,1,0,1,1,0,1,0,1,0,1),
(0,1,1,0,0,1,1,0,1,0,1,0,1),
(0,0,1,1,0,1,1,0,1,1,0,0,1),
\\&(1,0,1,1,0,1,1,1,0,0,0,0,0),
(0,0,1,1,0,1,1,1,0,1,0,0,0),
(0,0,1,0,0,1,1,1,0,1,0,0,1),
(1,1,1,0,0,1,1,1,0,1,0,1,1),
\\&(1,1,1,1,0,1,1,1,1,0,0,0,0),
(1,0,1,1,0,1,1,1,1,0,0,0,1),
(0,0,1,0,0,1,1,1,1,0,0,0,0),
(0,1,1,1,0,1,1,1,1,0,0,0,1),
\\&(1,0,1,0,0,1,1,1,1,0,1,0,0),
(0,0,0,0,0,1,1,1,1,0,1,0,0),
(1,1,0,1,0,1,1,1,1,1,0,0,0),
(1,1,1,0,0,1,1,1,1,1,0,0,1),
\\&(0,1,1,1,0,1,1,1,1,1,0,1,1),
(0,1,0,1,1,0,0,1,1,0,1,1,0),
(0,1,1,1,1,0,0,1,1,1,1,0,0),
(1,1,1,1,1,0,0,1,1,1,1,1,0),
\\&(0,1,0,1,1,0,1,0,0,0,1,0,0),
(1,1,0,1,1,0,1,0,1,0,0,0,1),
(0,0,1,0,1,0,1,1,0,0,0,0,0),
(0,0,1,1,1,0,1,1,1,1,1,0,0),
\\&(1,1,1,1,1,1,0,0,1,1,0,0,1),
(0,0,0,1,1,1,0,1,0,0,0,0,0),
(1,0,1,0,1,1,0,1,0,0,0,0,0),
(0,1,1,1,1,1,0,1,0,0,1,0,0),
\\&(0,0,1,1,1,1,0,1,0,1,0,0,1),
(0,0,0,1,1,1,0,1,0,1,0,1,0),
(1,0,0,1,1,1,0,1,1,0,0,1,0),
(0,0,0,0,1,1,0,1,1,0,1,0,0),
\\&(0,0,1,1,1,1,0,1,1,0,1,0,0),
(0,0,1,0,1,1,0,1,1,0,1,1,0),
(1,0,0,0,1,1,0,1,1,0,1,1,0),
(0,0,0,1,1,1,0,1,1,1,0,0,0),
\\&(1,0,0,0,1,1,0,1,1,1,0,0,0),
(0,1,1,1,1,1,0,1,1,1,0,1,1),
(0,1,0,0,1,1,1,0,0,0,0,0,0),
(1,0,1,0,1,1,1,0,0,0,1,1,0),
\\&(1,1,1,1,1,1,1,0,0,1,0,0,1),
(1,1,0,1,1,1,1,0,1,0,0,0,0),
(0,1,1,1,1,1,1,0,1,0,0,1,1),
(0,1,1,1,1,1,1,0,1,1,0,0,0),
\\&(0,0,1,1,1,1,1,1,0,0,0,0,0),
(1,0,1,0,1,1,1,1,0,0,0,0,1),
(0,1,0,1,0,1,0,1,1,0,1,1,0),
(0,0,0,1,1,1,1,1,0,1,0,0,0),
\\&(0,0,1,1,1,1,1,1,0,1,1,0,0),
(0,1,0,1,0,1,0,1,1,1,0,0,0),
(1,0,1,1,1,1,1,1,1,0,0,0,0),
(0,0,0,0,1,1,1,1,1,0,0,0,0),
\\&(0,0,1,0,1,1,1,1,1,0,0,0,1),
(1,1,1,1,1,1,1,1,1,0,0,1,1),
(0,0,1,0,1,1,1,1,1,0,1,0,0),
(0,1,1,1,0,1,0,1,1,0,0,1,1),
\\&(1,0,1,0,1,1,1,1,1,1,1,0,1),
(0,1,1,0,0,1,0,1,1,1,0,0,0)
\}.
\end{align*}}
Clearly, $|S_4^{13}|=122$. For a verification program showing that $S_4^{13}$ percolates under the $4$-neighbour bootstrap process in $Q_{13}$, see this \href{https://colab.research.google.com/drive/1_owaZNVjQTtzwKwdJ38Z51VGLKjKb-Fd?usp=sharing}{Google Colab Notebook} and the ancillary file \texttt{FourNeighbourBootstrapVerification.ipynb} uploaded with the arxiv version of this paper. 
\end{proof}

The next construction is for $d=14$. This matches the lower bound in Theorem~\ref{th:MN}.

\begin{lem}
\label{lem:Q14}
There exists a percolating set $S_4^{14}$ of cardinality $148$ for the $4$-neighbour bootstrap process in $Q_{14}$. 
\end{lem}

\begin{proof}
Define
{\tiny
\begin{align*}
S_4^{14}:=
\{
&(0, 0, 1, 0, 0, 0, 0, 0, 0, 0, 0, 0, 0, 0),
(0, 0, 1, 1, 0, 0, 0, 0, 0, 0, 0, 0, 0, 1),
(0, 0, 0, 1, 0, 0, 0, 0, 0, 0, 0, 0, 1, 0),
(0, 1, 1, 1, 1, 0, 0, 0, 0, 0, 0, 0, 0, 1),
\\&(0, 1, 1, 0, 1, 0, 0, 0, 0, 0, 0, 1, 0, 0),
(1, 0, 1, 0, 0, 0, 0, 0, 0, 0, 0, 1, 0, 0),
(0, 1, 1, 1, 1, 0, 0, 0, 0, 0, 0, 1, 1, 0),
(1, 0, 0, 1, 0, 0, 0, 0, 0, 0, 0, 1, 0, 0),
\\&(1, 0, 0, 1, 0, 0, 0, 0, 0, 0, 1, 0, 0, 0),
(0, 0, 0, 0, 1, 0, 0, 0, 0, 0, 1, 0, 0, 1),
(1, 1, 0, 0, 0, 0, 0, 0, 0, 0, 1, 0, 1, 0),
(0, 0, 0, 1, 0, 0, 0, 0, 0, 0, 1, 0, 1, 1),
\\&(1, 0, 0, 0, 0, 0, 0, 0, 0, 0, 1, 1, 0, 0),
(0, 1, 0, 0, 0, 0, 0, 0, 0, 0, 1, 1, 0, 1),
(0, 0, 0, 0, 0, 0, 0, 0, 0, 0, 1, 1, 1, 0),
(0, 1, 1, 0, 0, 0, 0, 0, 0, 0, 1, 1, 0, 0),
\\&(0, 0, 0, 1, 1, 0, 0, 0, 0, 1, 0, 0, 0, 0),
(1, 1, 0, 1, 1, 0, 0, 0, 0, 1, 0, 0, 0, 0),
(1, 0, 0, 1, 0, 0, 0, 0, 0, 1, 0, 0, 0, 0),
(0, 1, 1, 0, 0, 0, 0, 0, 0, 0, 1, 0, 1, 0),
\\&(1, 0, 0, 0, 0, 0, 0, 0, 0, 1, 0, 1, 0, 0),
(1, 0, 0, 1, 1, 0, 0, 0, 0, 1, 0, 1, 0, 1),
(0, 0, 1, 0, 0, 0, 0, 0, 0, 1, 0, 1, 0, 1),
(0, 0, 0, 1, 0, 0, 0, 0, 0, 1, 1, 0, 0, 1),
\\&(0, 0, 0, 0, 0, 0, 0, 0, 0, 1, 1, 1, 0, 1),
(1, 1, 0, 0, 0, 0, 0, 0, 1, 0, 0, 0, 0, 1),
(1, 1, 0, 0, 0, 0, 0, 0, 1, 0, 0, 0, 1, 0),
(0, 0, 0, 0, 1, 0, 0, 0, 1, 0, 0, 0, 0, 1),
\\&(1, 0, 0, 0, 1, 0, 0, 0, 1, 0, 0, 1, 0, 0),
(0, 1, 1, 0, 0, 0, 0, 0, 1, 0, 0, 0, 1, 0),
(0, 1, 1, 0, 0, 0, 0, 0, 1, 0, 0, 1, 0, 0),
(0, 1, 0, 0, 0, 0, 0, 0, 1, 0, 0, 1, 1, 0),
\\&(1, 0, 0, 0, 1, 0, 0, 0, 1, 0, 1, 0, 0, 0),
(0, 0, 0, 0, 0, 0, 0, 0, 1, 0, 1, 0, 0, 1),
(0, 1, 1, 1, 0, 0, 0, 0, 1, 0, 0, 1, 1, 1),
(0, 0, 1, 0, 0, 0, 0, 0, 1, 0, 1, 0, 1, 1),
\\&(0, 0, 0, 0, 0, 0, 0, 0, 1, 1, 0, 0, 1, 1),
(0, 0, 0, 0, 0, 0, 0, 0, 1, 1, 0, 1, 0, 1),
(0, 0, 1, 0, 1, 0, 0, 0, 1, 1, 1, 1, 1, 0),
(1, 0, 1, 0, 0, 0, 0, 1, 0, 0, 0, 0, 0, 0),
\\&(0, 0, 0, 0, 0, 0, 0, 1, 0, 0, 0, 0, 0, 1),
(0, 1, 0, 0, 0, 0, 0, 1, 0, 0, 0, 0, 1, 0),
(0, 0, 0, 1, 0, 0, 0, 1, 0, 0, 0, 1, 0, 0),
(1, 1, 0, 0, 0, 0, 0, 1, 0, 0, 0, 1, 0, 0),
\\&(0, 0, 0, 0, 0, 0, 0, 1, 0, 0, 0, 1, 1, 0),
(0, 0, 0, 0, 0, 0, 0, 1, 0, 0, 1, 0, 1, 0),
(0, 0, 1, 0, 0, 0, 0, 1, 0, 0, 1, 1, 0, 0),
(1, 1, 1, 0, 0, 0, 0, 1, 0, 0, 1, 1, 0, 1),
\\&(1, 0, 0, 0, 0, 0, 0, 1, 0, 0, 1, 1, 1, 0),
(0, 0, 1, 1, 0, 0, 0, 1, 0, 0, 1, 1, 1, 0),
(0, 0, 1, 0, 0, 0, 0, 0, 0, 0, 0, 0, 1, 1),
(0, 0, 0, 0, 0, 0, 0, 1, 0, 1, 0, 1, 0, 0),
\\&(1, 0, 0, 1, 0, 0, 0, 1, 0, 1, 0, 1, 0, 0),
(0, 1, 0, 0, 0, 0, 0, 1, 1, 0, 0, 0, 0, 1),
(0, 0, 0, 0, 0, 0, 0, 1, 1, 0, 0, 0, 1, 0),
(0, 0, 0, 0, 0, 0, 0, 1, 1, 0, 0, 1, 0, 1),
\\&(0, 0, 0, 0, 0, 0, 0, 1, 1, 0, 1, 0, 1, 1),
(0, 1, 0, 0, 0, 0, 0, 1, 1, 1, 0, 0, 0, 0),
(0, 0, 0, 0, 0, 0, 0, 1, 1, 1, 0, 0, 0, 1),
(0, 1, 0, 0, 1, 0, 0, 1, 1, 1, 0, 0, 1, 0),
\\&(0, 0, 0, 0, 0, 0, 0, 1, 1, 1, 1, 0, 0, 0),
(1, 0, 0, 0, 0, 0, 1, 0, 0, 0, 0, 0, 0, 0),
(0, 0, 0, 0, 0, 0, 1, 0, 0, 0, 0, 0, 0, 1),
(0, 1, 0, 0, 0, 0, 1, 0, 0, 0, 0, 0, 1, 0),
\\&(1, 1, 1, 0, 0, 0, 1, 0, 0, 0, 0, 0, 1, 0),
(0, 1, 1, 0, 0, 0, 1, 0, 0, 0, 0, 1, 0, 0),
(1, 0, 1, 1, 0, 0, 1, 0, 0, 0, 0, 0, 1, 0),
(0, 0, 1, 1, 0, 0, 1, 0, 0, 0, 0, 0, 1, 1),
\\&(0, 0, 1, 1, 1, 0, 1, 0, 0, 0, 0, 0, 0, 1),
(0, 0, 0, 0, 0, 0, 1, 0, 0, 0, 1, 0, 0, 0),
(0, 1, 0, 0, 1, 0, 1, 0, 0, 0, 1, 0, 0, 1),
(1, 1, 0, 0, 1, 0, 1, 0, 0, 0, 0, 1, 0, 0),
\\&(0, 1, 1, 1, 1, 0, 1, 0, 0, 0, 1, 1, 0, 0),
(0, 0, 1, 1, 0, 0, 1, 0, 0, 0, 1, 1, 1, 0),
(1, 1, 0, 1, 1, 0, 1, 0, 0, 1, 0, 0, 0, 1),
(0, 0, 1, 1, 1, 0, 0, 0, 0, 0, 1, 0, 0, 1),
\\&(0, 0, 0, 0, 0, 0, 1, 0, 1, 0, 0, 0, 0, 0),
(0, 0, 1, 0, 1, 0, 1, 0, 1, 0, 0, 0, 0, 0),
(0, 1, 0, 0, 1, 0, 1, 0, 1, 0, 0, 1, 0, 1),
(0, 1, 0, 1, 0, 0, 1, 0, 1, 1, 0, 0, 0, 0),
\\&(0, 1, 1, 0, 0, 0, 1, 1, 0, 0, 0, 0, 1, 0),
(1, 0, 0, 0, 0, 0, 1, 1, 0, 0, 0, 1, 0, 0),
(0, 0, 1, 1, 0, 0, 1, 1, 0, 0, 0, 1, 1, 0),
(0, 0, 1, 0, 0, 0, 1, 1, 0, 0, 1, 0, 1, 0),
\\&(0, 1, 1, 1, 1, 0, 1, 1, 0, 1, 0, 0, 0, 0),
(0, 1, 1, 1, 0, 0, 0, 0, 0, 0, 0, 0, 0, 0),
(0, 0, 1, 0, 0, 0, 1, 1, 1, 0, 1, 1, 0, 0),
(1, 1, 0, 0, 1, 0, 1, 1, 1, 1, 1, 0, 0, 0),
\\&(0, 0, 1, 1, 0, 0, 1, 1, 1, 1, 1, 0, 1, 0),
(0, 1, 1, 1, 1, 1, 0, 0, 0, 0, 0, 0, 0, 0),
(1, 1, 1, 0, 0, 1, 0, 0, 0, 0, 0, 0, 0, 1),
(0, 0, 1, 0, 0, 1, 0, 0, 0, 0, 0, 0, 0, 1),
\\&(0, 1, 1, 0, 0, 1, 0, 0, 0, 0, 0, 0, 1, 0),
(0, 0, 1, 0, 0, 1, 0, 0, 0, 0, 0, 1, 0, 0),
(0, 1, 1, 1, 1, 1, 0, 0, 0, 0, 0, 0, 1, 1),
(1, 0, 0, 0, 0, 1, 0, 0, 0, 0, 0, 1, 1, 0),
\\&(0, 0, 0, 1, 0, 1, 0, 0, 0, 0, 0, 1, 0, 0),
(1, 0, 0, 0, 0, 1, 0, 0, 0, 0, 1, 0, 0, 0),
(0, 1, 0, 1, 1, 1, 0, 0, 0, 0, 0, 1, 1, 0),
(1, 1, 1, 1, 0, 1, 0, 0, 0, 0, 0, 0, 1, 0),
\\&(0, 0, 1, 1, 0, 1, 0, 0, 0, 0, 1, 1, 0, 0),
(0, 0, 1, 1, 1, 1, 0, 0, 0, 1, 0, 0, 0, 1),
(0, 0, 1, 0, 0, 1, 0, 0, 0, 1, 0, 0, 1, 0),
(0, 0, 0, 0, 0, 1, 0, 0, 0, 1, 0, 0, 1, 1),
\\&(0, 0, 0, 0, 0, 1, 0, 0, 0, 1, 0, 1, 0, 1),
(0, 0, 1, 0, 0, 1, 0, 0, 0, 1, 1, 0, 0, 1),
(1, 1, 1, 1, 1, 1, 0, 0, 0, 1, 1, 1, 1, 1),
(0, 1, 1, 1, 0, 1, 0, 0, 1, 0, 0, 0, 0, 0),
\\&(1, 1, 0, 1, 0, 1, 0, 0, 1, 0, 0, 0, 0, 0),
(0, 0, 1, 0, 0, 1, 0, 0, 1, 0, 0, 0, 1, 0),
(1, 1, 1, 1, 0, 1, 0, 0, 1, 0, 0, 0, 1, 1),
(0, 0, 0, 0, 0, 1, 0, 0, 1, 0, 0, 1, 1, 0),
\\&(0, 0, 0, 0, 0, 1, 0, 0, 1, 1, 0, 0, 0, 1),
(1, 1, 0, 1, 0, 1, 0, 0, 1, 1, 0, 1, 0, 0),
(0, 0, 0, 0, 0, 1, 0, 0, 1, 1, 1, 0, 0, 0),
(0, 0, 1, 1, 0, 1, 0, 0, 1, 1, 1, 0, 0, 1),
\\&(0, 1, 0, 1, 1, 0, 1, 0, 0, 0, 0, 0, 0, 1),
(1, 1, 0, 0, 1, 1, 0, 1, 1, 0, 0, 0, 0, 0),
(1, 1, 0, 0, 0, 1, 0, 1, 1, 0, 0, 0, 0, 1),
(0, 1, 0, 0, 1, 1, 0, 1, 1, 0, 0, 1, 1, 1),
\\&(0, 1, 1, 0, 0, 1, 1, 0, 0, 0, 0, 0, 0, 0),
(1, 0, 1, 0, 0, 1, 1, 0, 0, 0, 0, 0, 0, 0),
(0, 1, 1, 1, 1, 1, 1, 0, 0, 0, 0, 0, 0, 1),
(1, 1, 1, 1, 0, 1, 1, 0, 0, 0, 0, 0, 0, 0),
\\&(0, 1, 0, 0, 1, 1, 1, 0, 0, 0, 0, 0, 0, 0),
(1, 1, 0, 1, 0, 1, 1, 0, 0, 0, 0, 0, 0, 1),
(1, 0, 0, 0, 1, 1, 1, 0, 0, 0, 0, 0, 0, 0),
(0, 0, 0, 1, 1, 1, 1, 0, 0, 0, 0, 0, 0, 0),
\\&(0, 1, 0, 0, 0, 1, 1, 0, 0, 0, 1, 0, 0, 0),
(1, 0, 0, 0, 0, 1, 1, 0, 0, 0, 0, 1, 0, 0),
(0, 0, 0, 0, 0, 1, 1, 0, 0, 0, 1, 0, 1, 1),
(1, 0, 1, 0, 1, 1, 1, 0, 0, 0, 1, 1, 0, 0),
\\&(0, 0, 0, 1, 0, 1, 1, 0, 0, 1, 0, 0, 0, 0),
(0, 1, 0, 1, 1, 1, 1, 0, 0, 1, 0, 0, 0, 0),
(0, 0, 1, 1, 1, 1, 1, 0, 0, 1, 0, 0, 1, 0),
(1, 1, 0, 1, 0, 0, 0, 0, 0, 0, 0, 0, 0, 0),
\\&(0, 0, 0, 0, 0, 1, 1, 0, 1, 0, 1, 0, 1, 0),
(1, 0, 0, 0, 1, 1, 1, 0, 1, 1, 0, 0, 0, 0),
(0, 0, 1, 1, 0, 1, 1, 0, 1, 1, 0, 0, 1, 1),
(1, 1, 0, 0, 1, 0, 0, 0, 0, 0, 0, 1, 1, 1),
\\&(1, 1, 1, 0, 0, 0, 0, 0, 0, 0, 0, 0, 0, 0),
(1, 0, 1, 0, 1, 0, 1, 0, 0, 0, 0, 0, 0, 0),
(1, 1, 1, 1, 1, 1, 1, 1, 0, 1, 0, 1, 1, 0),
(0, 0, 1, 0, 1, 1, 1, 0, 0, 0, 0, 0, 0, 0),
\\&(1, 1, 1, 0, 1, 1, 1, 1, 1, 0, 0, 0, 0, 0),
(0, 0, 1, 0, 1, 1, 1, 1, 1, 0, 0, 0, 0, 0),
(1, 1, 0, 1, 1, 1, 1, 1, 1, 0, 0, 1, 0, 1),
(1, 1, 0, 0, 1, 1, 1, 1, 1, 0, 1, 0, 0, 0)
\}.
\end{align*}}
Clearly, $|S_4^{14}|=148$. For a verification program showing that $S_4^{14}$ percolates under the $4$-neighbour bootstrap process in $Q_{14}$, see this \href{https://colab.research.google.com/drive/1_owaZNVjQTtzwKwdJ38Z51VGLKjKb-Fd?usp=sharing}{Google Colab Notebook} and the ancillary file \texttt{FourNeighbourBootstrapVerification.ipynb} uploaded with the arxiv version of this paper. 
\end{proof}

The next construction is for $d=15$. This matches the lower bound in Theorem~\ref{th:MN}.

\begin{lem}
\label{lem:Q15}
There exists a percolating set $S_4^{15}$ of cardinality $179$ for the $4$-neighbour bootstrap process in $Q_{15}$. 
\end{lem}

\begin{proof}
Define
{\tiny\begin{align*}
S_4^{15}:=
\{
&(0, 0, 0, 1, 0, 0, 0, 0, 0, 0, 0, 0, 0, 0, 0),
(0, 0, 1, 1, 1, 0, 0, 0, 0, 0, 0, 0, 0, 0, 0),
(0, 0, 1, 0, 0, 1, 0, 0, 0, 0, 0, 0, 0, 1, 0),
(0, 1, 0, 0, 1, 0, 0, 0, 0, 0, 0, 0, 0, 0, 0),
\\&(0, 1, 0, 1, 0, 1, 0, 0, 0, 0, 0, 0, 1, 0, 0),
(0, 0, 1, 0, 1, 0, 0, 0, 0, 0, 0, 0, 1, 0, 1),
(0, 0, 1, 0, 0, 0, 0, 0, 0, 0, 0, 0, 1, 0, 0),
(1, 1, 0, 0, 0, 1, 0, 0, 0, 0, 0, 0, 0, 1, 1),
\\&(0, 1, 1, 0, 0, 0, 0, 0, 0, 0, 0, 1, 0, 0, 0),
(0, 1, 0, 0, 0, 0, 0, 0, 0, 0, 0, 1, 0, 0, 1),
(0, 0, 1, 0, 1, 1, 0, 0, 0, 0, 0, 1, 0, 1, 0),
(0, 0, 0, 1, 1, 1, 0, 0, 0, 0, 0, 1, 0, 0, 0),
\\&(0, 0, 1, 0, 0, 1, 0, 0, 0, 0, 0, 1, 0, 0, 0),
(0, 1, 0, 0, 0, 1, 0, 0, 0, 0, 0, 1, 1, 0, 1),
(0, 0, 1, 0, 1, 1, 0, 0, 0, 0, 1, 0, 0, 0, 0),
(0, 0, 1, 0, 0, 1, 0, 0, 0, 0, 1, 0, 0, 0, 1),
\\&(0, 0, 0, 0, 1, 1, 0, 0, 0, 0, 1, 0, 0, 1, 0),
(1, 1, 0, 1, 1, 0, 0, 0, 0, 0, 1, 0, 0, 0, 0),
(0, 0, 1, 0, 1, 0, 0, 0, 0, 0, 1, 0, 1, 0, 0),
(0, 0, 1, 0, 0, 1, 0, 0, 0, 0, 1, 0, 1, 0, 0),
\\&(0, 0, 0, 0, 1, 1, 0, 0, 0, 0, 1, 1, 0, 0, 0),
(1, 1, 1, 1, 0, 1, 0, 0, 0, 1, 0, 0, 0, 0, 0),
(1, 1, 0, 1, 0, 0, 0, 0, 0, 1, 0, 0, 0, 0, 0),
(0, 0, 1, 0, 0, 0, 0, 0, 0, 1, 0, 0, 0, 0, 0),
\\&(0, 0, 1, 1, 0, 0, 0, 0, 0, 1, 0, 0, 0, 0, 1),
(0, 0, 1, 0, 1, 0, 0, 0, 0, 1, 0, 0, 0, 1, 0),
(0, 0, 0, 0, 0, 0, 0, 0, 0, 1, 0, 0, 1, 0, 1),
(1, 0, 0, 1, 0, 0, 0, 0, 0, 1, 0, 0, 0, 1, 1),
\\&(0, 1, 0, 0, 0, 0, 0, 0, 0, 1, 0, 1, 0, 0, 0),
(0, 0, 0, 0, 1, 0, 0, 0, 0, 0, 0, 0, 0, 1, 0),
(0, 0, 0, 0, 0, 0, 0, 0, 0, 1, 0, 1, 1, 0, 0),
(0, 0, 0, 0, 0, 0, 0, 0, 0, 1, 1, 0, 0, 0, 1),
\\&(0, 0, 1, 0, 1, 0, 0, 0, 0, 1, 1, 0, 1, 0, 1),
(0, 0, 1, 0, 1, 1, 0, 0, 1, 0, 0, 0, 0, 0, 0),
(0, 0, 1, 1, 1, 0, 0, 0, 1, 0, 0, 0, 0, 0, 1),
(0, 0, 1, 1, 1, 0, 0, 0, 1, 0, 0, 0, 0, 1, 0),
\\&(0, 0, 0, 1, 0, 0, 0, 0, 1, 0, 0, 0, 0, 0, 1),
(0, 0, 0, 0, 1, 1, 0, 0, 1, 0, 0, 0, 0, 1, 0),
(0, 0, 0, 0, 0, 1, 0, 0, 1, 0, 0, 0, 1, 0, 0),
(0, 1, 0, 0, 0, 0, 0, 0, 1, 0, 0, 0, 0, 0, 1),
\\&(0, 0, 0, 0, 0, 0, 0, 0, 1, 0, 0, 0, 1, 1, 0),
(0, 1, 0, 0, 1, 0, 0, 0, 1, 0, 0, 0, 0, 1, 0),
(0, 1, 0, 0, 1, 0, 0, 0, 1, 0, 0, 1, 0, 0, 0),
(0, 0, 0, 0, 1, 0, 0, 0, 1, 0, 0, 0, 0, 0, 1),
\\&(0, 0, 0, 0, 0, 0, 0, 0, 1, 0, 0, 1, 0, 1, 0),
(0, 0, 0, 1, 1, 0, 0, 0, 1, 0, 0, 1, 0, 1, 0),
(0, 0, 0, 0, 0, 0, 0, 0, 1, 0, 1, 1, 0, 0, 0),
(0, 1, 1, 0, 0, 0, 0, 0, 1, 1, 0, 0, 0, 0, 1),
\\&(0, 0, 0, 0, 0, 0, 0, 0, 1, 1, 0, 0, 0, 1, 0),
(0, 0, 0, 0, 0, 0, 0, 0, 1, 1, 0, 0, 1, 0, 0),
(0, 1, 0, 0, 0, 0, 0, 0, 1, 1, 0, 1, 0, 1, 0),
(0, 0, 1, 1, 1, 0, 0, 0, 1, 1, 0, 1, 0, 1, 0),
\\&(0, 0, 0, 0, 0, 0, 0, 0, 1, 1, 1, 0, 0, 0, 0),
(0, 0, 0, 1, 0, 0, 0, 0, 1, 1, 1, 0, 1, 1, 0),
(0, 0, 0, 1, 1, 1, 0, 0, 0, 0, 0, 0, 0, 1, 0),
(1, 0, 1, 1, 1, 0, 0, 1, 0, 0, 0, 0, 0, 0, 0),
\\&(0, 0, 0, 0, 0, 0, 0, 1, 0, 0, 0, 0, 0, 0, 1),
(1, 1, 1, 1, 0, 1, 0, 1, 0, 0, 0, 0, 0, 1, 0),
(0, 0, 0, 1, 1, 1, 0, 1, 0, 0, 0, 0, 0, 0, 0),
(1, 0, 1, 1, 0, 0, 0, 1, 0, 0, 0, 0, 1, 0, 0),
\\&(1, 0, 1, 0, 0, 0, 0, 1, 0, 0, 0, 0, 1, 0, 1),
(0, 1, 0, 1, 0, 0, 0, 1, 0, 0, 0, 0, 1, 1, 0),
(1, 0, 0, 0, 0, 1, 0, 1, 0, 0, 0, 0, 1, 1, 1),
(0, 0, 0, 0, 1, 0, 0, 1, 0, 0, 0, 0, 1, 0, 1),
\\&(0, 1, 0, 0, 1, 0, 0, 1, 0, 0, 0, 1, 0, 0, 0),
(0, 1, 0, 0, 0, 0, 0, 1, 0, 0, 0, 1, 0, 1, 0),
(0, 1, 0, 0, 1, 0, 0, 1, 0, 0, 0, 1, 0, 1, 1),
(0, 0, 0, 1, 1, 1, 0, 1, 0, 0, 0, 1, 1, 0, 0),
\\&(1, 0, 1, 1, 1, 0, 0, 1, 0, 0, 0, 0, 1, 0, 1),
(0, 0, 1, 0, 1, 0, 0, 1, 0, 0, 1, 0, 0, 0, 0),
(0, 0, 0, 0, 1, 0, 0, 1, 0, 0, 1, 0, 0, 0, 1),
(0, 0, 0, 0, 1, 0, 0, 1, 0, 0, 1, 0, 0, 1, 0),
\\&(0, 1, 0, 0, 0, 0, 0, 1, 0, 0, 1, 0, 0, 0, 1),
(0, 0, 0, 0, 0, 0, 0, 1, 0, 0, 1, 0, 1, 1, 0),
(0, 0, 0, 0, 0, 0, 0, 1, 0, 0, 1, 1, 0, 0, 0),
(0, 0, 0, 1, 0, 0, 0, 1, 0, 1, 0, 0, 0, 0, 0),
\\&(1, 0, 1, 1, 1, 0, 0, 1, 0, 1, 0, 0, 0, 0, 1),
(1, 1, 0, 0, 0, 1, 0, 1, 0, 1, 0, 0, 0, 1, 1),
(1, 1, 0, 0, 0, 0, 0, 1, 0, 1, 0, 0, 1, 1, 0),
(1, 1, 0, 1, 0, 1, 0, 1, 0, 1, 0, 0, 1, 1, 0),
\\&(0, 0, 0, 0, 0, 0, 0, 1, 0, 1, 0, 1, 0, 0, 0),
(1, 0, 0, 0, 0, 1, 0, 1, 0, 1, 0, 0, 1, 1, 0),
(1, 0, 0, 0, 0, 1, 0, 1, 0, 1, 0, 1, 0, 0, 0),
(1, 1, 0, 1, 1, 0, 0, 1, 0, 1, 1, 0, 0, 0, 1),
\\&(0, 0, 0, 0, 0, 0, 0, 1, 0, 1, 1, 1, 0, 1, 0),
(0, 0, 0, 1, 0, 0, 0, 1, 1, 0, 0, 0, 0, 0, 0),
(0, 0, 0, 1, 1, 1, 0, 1, 1, 0, 0, 0, 0, 0, 1),
(0, 1, 0, 0, 0, 0, 0, 1, 1, 0, 0, 0, 0, 1, 0),
\\&(1, 0, 1, 1, 1, 1, 0, 1, 1, 0, 0, 0, 0, 1, 1),
(1, 0, 0, 1, 1, 0, 0, 1, 1, 0, 0, 1, 1, 0, 1),
(1, 0, 1, 1, 0, 1, 0, 1, 1, 0, 1, 0, 0, 1, 0),
(0, 0, 0, 0, 1, 0, 0, 1, 1, 0, 1, 0, 1, 1, 0),
\\&(1, 0, 1, 1, 0, 1, 0, 1, 1, 0, 1, 1, 0, 0, 0),
(0, 1, 0, 0, 1, 0, 0, 1, 1, 0, 1, 1, 0, 0, 1),
(0, 0, 0, 0, 0, 1, 0, 1, 1, 1, 0, 0, 0, 1, 0),
(1, 0, 1, 0, 1, 1, 0, 1, 1, 1, 0, 0, 0, 1, 1),
\\&(0, 1, 0, 0, 1, 0, 0, 1, 1, 1, 0, 1, 0, 1, 1),
(0, 0, 0, 0, 0, 0, 0, 1, 1, 1, 0, 1, 1, 0, 0),
(0, 0, 0, 0, 0, 0, 0, 1, 1, 1, 1, 0, 0, 1, 0),
(0, 1, 1, 0, 1, 0, 1, 0, 0, 0, 0, 0, 0, 0, 0),
\\&(0, 1, 0, 1, 0, 0, 1, 0, 0, 0, 0, 0, 0, 0, 0),
(0, 1, 0, 0, 0, 0, 1, 0, 0, 0, 0, 0, 0, 0, 1),
(0, 1, 0, 1, 0, 1, 1, 0, 0, 0, 0, 0, 0, 1, 1),
(0, 1, 0, 0, 1, 0, 1, 0, 0, 0, 0, 0, 0, 1, 0),
\\&(1, 1, 0, 1, 0, 1, 1, 0, 0, 0, 0, 0, 0, 0, 0),
(0, 0, 0, 0, 0, 0, 1, 0, 0, 0, 0, 0, 1, 0, 1),
(0, 0, 0, 0, 1, 0, 1, 0, 0, 0, 0, 0, 1, 0, 0),
(0, 0, 1, 0, 1, 1, 1, 0, 0, 0, 0, 0, 1, 0, 0),
\\&(1, 0, 0, 0, 1, 0, 1, 0, 0, 0, 0, 1, 0, 0, 1),
(0, 0, 0, 0, 1, 0, 1, 0, 0, 0, 0, 1, 0, 0, 0),
(0, 1, 1, 0, 1, 1, 1, 0, 0, 0, 0, 1, 0, 1, 1),
(0, 0, 0, 1, 1, 1, 1, 0, 0, 0, 0, 1, 1, 0, 0),
\\&(0, 1, 0, 0, 1, 0, 1, 0, 0, 0, 0, 1, 0, 1, 1),
(0, 0, 0, 0, 1, 0, 1, 0, 0, 0, 1, 0, 0, 0, 0),
(1, 0, 0, 1, 1, 1, 1, 0, 0, 0, 1, 0, 1, 0, 0),
(1, 0, 1, 0, 1, 1, 1, 0, 0, 0, 1, 0, 1, 0, 0),
\\&(0, 0, 1, 0, 1, 1, 1, 0, 0, 0, 1, 0, 1, 1, 0),
(0, 1, 0, 0, 0, 1, 0, 0, 0, 0, 0, 0, 0, 0, 1),
(0, 0, 0, 0, 1, 0, 1, 0, 0, 0, 1, 0, 1, 0, 1),
(0, 0, 0, 1, 1, 0, 1, 0, 0, 0, 1, 1, 0, 0, 1),
\\&(0, 0, 0, 0, 1, 0, 1, 0, 0, 0, 1, 1, 1, 0, 0),
(0, 0, 1, 1, 1, 1, 1, 0, 0, 0, 1, 1, 1, 0, 0),
(1, 1, 0, 0, 0, 0, 1, 0, 0, 1, 0, 0, 0, 0, 1),
(0, 0, 0, 0, 0, 0, 1, 0, 0, 1, 0, 0, 0, 1, 0),
\\&(1, 1, 0, 0, 1, 0, 0, 0, 0, 0, 0, 0, 0, 0, 1),
(0, 0, 0, 0, 0, 0, 1, 0, 0, 1, 1, 0, 0, 0, 0),
(0, 0, 0, 0, 1, 1, 1, 0, 0, 0, 0, 0, 0, 0, 1),
(0, 0, 0, 1, 1, 1, 1, 0, 1, 0, 0, 0, 0, 1, 0),
\\&(1, 0, 1, 1, 1, 1, 1, 0, 1, 0, 0, 1, 0, 0, 0),
(0, 0, 1, 1, 1, 1, 1, 0, 1, 0, 0, 1, 0, 1, 0),
(0, 0, 0, 1, 0, 0, 1, 0, 0, 0, 0, 0, 0, 0, 1),
(1, 1, 0, 1, 1, 1, 1, 0, 1, 0, 1, 0, 0, 1, 0),
\\&(0, 0, 0, 0, 1, 1, 1, 0, 1, 0, 1, 1, 0, 1, 0),
(0, 0, 0, 1, 1, 0, 1, 0, 0, 0, 0, 0, 0, 0, 0),
(1, 0, 1, 1, 0, 0, 1, 1, 0, 0, 0, 0, 0, 0, 0),
(0, 0, 1, 0, 1, 0, 1, 1, 0, 0, 0, 0, 0, 0, 0),
\\&(1, 0, 0, 1, 1, 1, 1, 1, 0, 0, 0, 0, 0, 0, 0),
(0, 0, 0, 1, 1, 1, 1, 1, 0, 0, 0, 0, 0, 0, 1),
(1, 0, 1, 1, 1, 1, 1, 1, 0, 0, 0, 0, 0, 0, 1),
(1, 1, 0, 0, 0, 0, 1, 1, 0, 0, 0, 0, 0, 0, 1),
\\&(1, 1, 0, 1, 1, 0, 1, 1, 0, 0, 0, 0, 0, 0, 0),
(0, 1, 1, 0, 0, 0, 1, 1, 0, 0, 0, 0, 0, 0, 0),
(0, 0, 1, 1, 1, 1, 1, 1, 0, 0, 0, 0, 0, 1, 1),
(0, 0, 0, 1, 1, 0, 1, 1, 0, 0, 0, 0, 1, 0, 1),
\\&(0, 1, 1, 0, 0, 0, 1, 1, 0, 0, 0, 0, 0, 1, 1),
(1, 1, 1, 1, 0, 0, 1, 1, 0, 0, 0, 0, 1, 0, 0),
(0, 1, 0, 0, 1, 1, 1, 1, 0, 0, 0, 1, 0, 0, 0),
(1, 0, 0, 0, 1, 1, 1, 1, 0, 0, 0, 1, 0, 0, 1),
\\&(0, 0, 0, 0, 1, 1, 1, 1, 0, 0, 0, 1, 1, 0, 0),
(0, 0, 0, 0, 0, 0, 1, 1, 0, 0, 0, 1, 1, 0, 0),
(0, 0, 0, 0, 0, 0, 1, 1, 0, 0, 0, 0, 0, 0, 0),
(1, 1, 0, 1, 0, 0, 1, 1, 0, 0, 1, 0, 0, 0, 0),
\\&(1, 1, 0, 1, 1, 1, 1, 1, 0, 0, 1, 0, 0, 0, 0),
(0, 0, 1, 0, 0, 0, 1, 0, 0, 0, 0, 0, 0, 1, 0),
(0, 0, 0, 1, 1, 1, 1, 1, 0, 0, 1, 0, 1, 0, 0),
(1, 0, 1, 1, 0, 1, 1, 1, 0, 0, 1, 0, 1, 0, 0),
\\&(0, 0, 0, 1, 0, 0, 1, 1, 0, 0, 1, 0, 1, 1, 0),
(0, 0, 1, 0, 0, 0, 1, 1, 0, 0, 1, 0, 0, 1, 1),
(0, 0, 1, 0, 0, 1, 1, 0, 0, 0, 0, 0, 0, 0, 0),
(0, 0, 0, 0, 0, 0, 1, 1, 0, 0, 1, 1, 0, 1, 0),
\\&(1, 1, 1, 1, 1, 0, 1, 1, 0, 1, 0, 0, 0, 0, 0),
(1, 1, 1, 1, 1, 1, 1, 1, 0, 1, 0, 0, 0, 0, 1),
(0, 1, 0, 0, 0, 0, 1, 1, 0, 1, 0, 0, 0, 0, 0),
(0, 0, 0, 1, 1, 0, 0, 0, 1, 0, 0, 0, 0, 0, 0),
\\&(1, 1, 0, 1, 0, 0, 1, 1, 0, 1, 0, 0, 0, 0, 0),
(0, 0, 0, 0, 0, 0, 1, 1, 0, 1, 0, 0, 1, 0, 1),
(1, 0, 1, 0, 1, 1, 1, 1, 0, 1, 0, 1, 0, 0, 0),
(0, 0, 1, 0, 1, 0, 1, 0, 0, 0, 0, 0, 0, 0, 1),
\\&(0, 1, 0, 1, 0, 0, 1, 1, 0, 1, 1, 0, 0, 0, 0),
(0, 1, 0, 0, 0, 0, 1, 1, 0, 1, 1, 0, 0, 0, 1),
(0, 0, 0, 0, 1, 0, 1, 1, 0, 0, 0, 0, 0, 0, 1),
(1, 0, 0, 0, 0, 0, 1, 1, 0, 0, 0, 1, 0, 0, 0),
\\&(1, 1, 1, 1, 0, 0, 0, 0, 0, 0, 0, 0, 0, 0, 0),
(1, 1, 1, 1, 1, 1, 1, 1, 0, 0, 0, 1, 1, 1, 0),
(0, 0, 0, 0, 1, 1, 1, 1, 0, 0, 0, 0, 0, 1, 1),
(1, 0, 1, 0, 1, 1, 1, 1, 1, 0, 1, 0, 1, 1, 0),
\\&(1, 0, 1, 1, 0, 1, 1, 1, 1, 0, 1, 1, 0, 1, 1),
(0, 1, 1, 1, 1, 0, 1, 1, 1, 0, 1, 1, 1, 0, 1),
(0, 1, 0, 0, 0, 1, 1, 1, 1, 1, 1, 1, 0, 1, 1)
\}.
\end{align*}}
Clearly, $|S_4^{15}|=179$. For a verification program showing that $S_4^{15}$ percolates under the $4$-neighbour bootstrap process in $Q_{15}$, see this \href{https://colab.research.google.com/drive/1_owaZNVjQTtzwKwdJ38Z51VGLKjKb-Fd?usp=sharing}{Google Colab Notebook} and the ancillary file \texttt{FourNeighbourBootstrapVerification.ipynb} uploaded with the arxiv version of this paper. 
\end{proof}

\section{4-Neighbour Meta Constructions}
\label{sec:constructionsMeta4}

The goal of this section is to present examples of percolating labelings for the $4$-meta bootstrap process on $Q_d$ for $d=4$ and $d=12$ that will play the role of $\ell_0$ in our applications of Lemma~\ref{lem:meta}.

\begin{lem}
\label{lem:Q4meta}
There exists a percolating labeling $\ell_0^4:V(Q_4)\to\{0,1,2,3,4\}$ for the $4$-meta bootstrap process in $Q_4$ such that there are
\begin{itemize}
    \item 1 vertex of label 1,
    \item 3 vertices of label 2,
    \item 3 vertices of label 3,
    \item 1 vertex of label 4,
\end{itemize}
\end{lem}

\begin{proof}
Define $\ell_0^4:V(Q_4)\to\{0,1,2,3,4\}$ by letting
\[\ell_0^4(0,0,1,1)=1,\]
\[\ell_0^4(0,0,0,0)=\ell_0^4(0,1,0,1)=\ell_0^4(1,0,0,1)=2,\]
\[\ell_0^4(1,0,1,0)=\ell_0^4(1,1,0,0)=\ell_0^4(1,1,1,1)=3,\]
\[\ell_0^4(0,1,1,0)=4,\]
and mapping all other vertices to $0$. For a verification program showing that $\ell_0^4$ percolates under the $4$-meta bootstrap process in $Q_{4}$, see this \href{https://colab.research.google.com/drive/1_owaZNVjQTtzwKwdJ38Z51VGLKjKb-Fd?usp=sharing}{Google Colab Notebook} and the ancillary file \texttt{FourNeighbourBootstrapVerification.ipynb} uploaded with the arxiv version of this paper. 
\end{proof}

\begin{lem}
\label{lem:Q12meta}
There exists a percolating labeling $\ell_0^{12}:V(Q_{12})\to\{0,1,2,3,4\}$ for the $4$-meta bootstrap process in $Q_{12}$ such that there are
\begin{itemize}
    \item 55 vertices of label 1,
    \item 33 vertices of label 2,
    \item 9 vertices of label 3,
    \item 1 vertex of label 4,
\end{itemize}
and all other vertices have label $0$. 
\end{lem}

\begin{proof}
Define $\ell_0^{12}:V(Q_{12})\to\{0,1,2,3,4\}$ by letting
{\footnotesize
\begin{align*}
\ell_0^{12}(0,0,0,0,1,0,0,0,0,0,0,1)=\ell_0^{12}(0,0,0,0,1,1,0,0,1,0,0,1)=\ell_0^{12}(0,0,0,1,0,0,1,0,1,0,0,1)\\
=\ell_0^{12}(0,0,0,1,1,0,0,0,0,1,1,0)=\ell_0^{12}(0,0,1,0,1,0,1,0,0,0,0,1)=\ell_0^{12}(0,0,1,1,0,0,0,1,0,0,0,1)\\
=\ell_0^{12}(0,0,1,1,1,0,1,0,0,0,1,1)=\ell_0^{12}(0,1,0,0,0,0,1,0,0,0,1,1)=\ell_0^{12}(0,1,0,0,0,0,1,0,0,1,0,0)\\
=\ell_0^{12}(0,1,0,0,0,1,1,0,0,0,0,0)=\ell_0^{12}(0,1,0,0,1,0,1,1,1,1,0,0)=\ell_0^{12}(0,1,0,0,1,1,0,1,0,0,0,1)\\
=\ell_0^{12}(0,1,0,0,1,1,0,1,1,0,1,1)=\ell_0^{12}(0,1,0,1,0,0,0,0,0,1,1,0)=\ell_0^{12}(0,1,0,1,0,0,1,0,0,0,1,0)\\
=\ell_0^{12}(0,1,1,0,0,1,0,0,1,0,1,0)=\ell_0^{12}(0,1,1,0,0,1,1,0,0,1,1,0)=\ell_0^{12}(0,1,1,1,0,0,0,0,0,0,0,1)\\
=\ell_0^{12}(0,1,1,1,0,1,0,0,1,0,0,1)=\ell_0^{12}(0,1,1,1,1,0,0,0,0,0,1,1)=\ell_0^{12}(1,0,0,0,0,0,1,0,1,0,1,1)\\
=\ell_0^{12}(1,0,0,0,0,1,0,0,1,0,0,1)=\ell_0^{12}(1,0,0,0,1,0,1,0,1,0,0,0)=\ell_0^{12}(1,0,0,0,1,0,1,1,1,0,0,1)\\
=\ell_0^{12}(1,0,0,1,0,0,1,0,0,0,1,1)=\ell_0^{12}(1,0,0,1,0,0,1,0,1,1,0,1)=\ell_0^{12}(1,0,0,1,0,1,1,0,0,1,0,1)\\
=\ell_0^{12}(1,0,0,1,1,0,0,0,0,0,0,0)=\ell_0^{12}(1,0,1,0,0,0,0,0,1,0,0,1)=\ell_0^{12}(1,0,1,0,0,1,0,0,0,0,1,0)\\
=\ell_0^{12}(1,0,1,0,0,1,1,0,0,0,1,1)=\ell_0^{12}(1,0,1,0,0,1,1,0,1,0,0,1)=\ell_0^{12}(1,0,1,0,1,0,0,1,0,0,0,0)\\
=\ell_0^{12}(1,0,1,1,0,0,0,0,0,0,1,0)=\ell_0^{12}(1,0,1,1,0,1,1,1,0,0,1,1)=\ell_0^{12}(1,0,1,1,1,0,0,0,0,1,0,0)\\
=\ell_0^{12}(1,1,0,0,1,1,0,0,0,1,0,1)=\ell_0^{12}(1,1,0,1,0,0,0,0,1,1,0,0)=\ell_0^{12}(1,1,0,1,0,0,0,1,0,1,0,0)\\
=\ell_0^{12}(1,1,0,1,0,0,1,0,1,1,1,1)=\ell_0^{12}(1,1,0,1,1,0,0,0,1,1,0,1)=\ell_0^{12}(1,1,1,0,0,0,0,0,1,0,1,1)\\
=\ell_0^{12}(1,1,1,0,0,1,0,1,0,0,1,0)=\ell_0^{12}(1,1,1,0,0,1,0,1,0,1,1,1)=\ell_0^{12}(1,1,1,0,0,1,1,1,1,1,1,1)\\
=\ell_0^{12}(1,1,1,0,1,0,0,0,1,0,1,0)=\ell_0^{12}(1,1,1,0,1,0,0,1,1,0,0,0)=\ell_0^{12}(1,1,1,0,1,0,0,1,1,1,0,1)\\
=\ell_0^{12}(1,1,1,0,1,1,0,0,1,0,0,0)=\ell_0^{12}(1,1,1,1,0,0,0,0,1,1,1,1)=\ell_0^{12}(1,1,1,1,0,0,0,1,1,1,0,1)\\
=\ell_0^{12}(1,1,1,1,0,1,1,1,0,0,0,1)=\ell_0^{12}(1,1,1,1,1,0,0,0,1,0,0,1)=\ell_0^{12}(1,1,1,1,1,0,1,0,0,1,1,1)\\
=\ell_0^{12}(1,1,1,1,1,1,0,1,0,1,0,1)=1,
\end{align*}
\begin{align*}
\ell_0^{12}(0,1,0,0,0,0,0,0,0,1,1,1)=\ell_0^{12}(0,1,0,0,0,0,0,0,1,0,1,1)=\ell_0^{12}(0,1,0,0,0,1,0,0,1,0,0,1)\\
=\ell_0^{12}(0,1,0,0,1,0,1,0,0,0,0,1)=\ell_0^{12}(0,1,1,0,0,1,0,0,0,0,0,1)=\ell_0^{12}(1,0,0,0,0,0,0,0,0,0,1,1)\\
=\ell_0^{12}(1,0,0,0,0,0,0,0,0,1,0,0)=\ell_0^{12}(1,0,0,0,0,0,0,0,1,0,0,0)=\ell_0^{12}(1,0,0,0,0,0,0,1,0,1,0,1)\\
=\ell_0^{12}(1,0,0,0,0,1,1,0,0,0,0,1)=\ell_0^{12}(1,0,0,0,1,0,0,0,0,1,0,1)=\ell_0^{12}(1,0,0,1,0,1,0,0,0,0,0,1)\\
=\ell_0^{12}(1,0,1,0,0,0,0,0,0,1,0,1)=\ell_0^{12}(1,1,0,0,0,0,0,0,0,0,0,1)=\ell_0^{12}(1,1,0,0,0,0,0,1,0,0,1,1)\\
=\ell_0^{12}(1,1,0,0,0,0,0,1,1,1,0,1)=\ell_0^{12}(1,1,0,0,0,1,0,0,0,0,1,0)=\ell_0^{12}(1,1,0,0,1,0,0,0,0,0,1,0)\\
=\ell_0^{12}(1,1,0,0,1,0,0,0,1,0,0,0)=\ell_0^{12}(1,1,0,0,1,0,0,1,0,1,0,1)=\ell_0^{12}(1,1,0,0,1,0,1,0,0,0,0,0)\\
=\ell_0^{12}(1,1,0,0,1,1,1,0,0,0,0,1)=\ell_0^{12}(1,1,0,1,0,0,0,0,0,1,1,1)=\ell_0^{12}(1,1,0,1,0,0,0,1,1,0,0,1)\\
=\ell_0^{12}(1,1,0,1,0,0,1,0,1,0,0,1)=\ell_0^{12}(1,1,0,1,0,1,0,0,0,0,0,0)=\ell_0^{12}(1,1,1,0,0,0,0,1,0,0,0,0)\\
=\ell_0^{12}(1,1,1,0,0,0,1,0,0,0,1,1)=\ell_0^{12}(1,1,1,0,0,0,1,0,1,0,0,1)=\ell_0^{12}(1,1,1,0,0,1,0,0,0,1,0,1)\\
=\ell_0^{12}(1,1,1,0,0,1,0,1,0,0,0,1)=\ell_0^{12}(1,1,1,0,1,0,0,0,0,0,0,1)=\ell_0^{12}(1,1,1,1,0,0,0,0,0,0,1,1)=2,
\end{align*}
\begin{align*}
\ell_0^{12}(0,0,0,0,0,0,0,1,0,0,0,1)=\ell_0^{12}(0,1,0,0,0,0,0,1,0,0,0,0)=\ell_0^{12}(0,1,0,1,0,0,0,0,0,0,1,1)\\
=\ell_0^{12}(0,1,0,1,0,1,0,0,0,0,0,1)=\ell_0^{12}(0,1,0,1,1,0,0,0,0,0,0,1)=\ell_0^{12}(1,0,0,0,0,0,0,1,0,0,0,0)\\
=\ell_0^{12}(1,1,0,0,0,0,1,1,0,1,0,1)=\ell_0^{12}(1,1,0,0,0,0,1,1,1,0,0,1)=\ell_0^{12}(1,1,1,0,0,0,1,1,0,0,0,1)=3,
\end{align*}
\begin{align*}
\ell_0^{12}(0,0,0,0,0,0,0,0,0,0,0,0)=4,
\end{align*}}
and mapping all other vertices to $0$. For a verification program showing that $\ell_0^{12}$ percolates under the $4$-meta bootstrap process in $Q_{12}$, see this \href{https://colab.research.google.com/drive/1_owaZNVjQTtzwKwdJ38Z51VGLKjKb-Fd?usp=sharing}{Google Colab Notebook} and the ancillary file \texttt{FourNeighbourBootstrapVerification.ipynb} uploaded with the arxiv version of this paper. 
\end{proof}

\section{Proofs of Main Theorems}
\label{sec:proofs}

We now present the proofs of the main theorems. 

\begin{proof}[Proof of Theorem~\ref{th:main}]
The lower bound follows from Theorem~\ref{th:MN} and so we focus on the upper bound. The cases $d=4$ and $6\leq d\leq 15$ are covered by the lemmas in Section~\ref{sec:constructionsSeed4}. So, suppose that $d\equiv 0,4\bmod 6$. We proceed by induction on $d$, where the base cases $d\in{4,6,12}$ are covered by Lemmas~\ref{lem:Q4},~\ref{lem:Q6} and~\ref{lem:Q12}, respectively. Let $d\geq16$ such that $d\equiv 0,4\bmod 6$. We divide the proof into two cases. 

\begin{casee}
$d\equiv 4\bmod 6$.
\end{casee}

By induction, there exists a percolating set $S_4^{d-4}$ of cardinality $\frac{(d-4)((d-4)^2+3(d-4)+14)}{24}+1$ for the $4$-neighbour bootstrap process in $Q_{d-4}$. Also, by Lemma~\ref{lem:r=3} and the fact that $d-4$ is divisible by $6$, there exist sets $S_1^{d-4}\subseteq S_2^{d-4}\subseteq S_3^{d-4}$ of vertices of $Q_{d-4}$ such that $|S_1^{d-4}|=1,|S_2^{d-4}|=\frac{d-4}{2}+1$ and $|S_3^{d-4}|=\frac{(d-4)(d-1)}{6}+1$ and $S_i^{d-4}$ percolates with respect to the $i$-neighbour bootstrap process on $Q_{d-4}$ for all $i\in\{1,2,3\}$. By feeding the labeling $\ell_0^4$ from Lemma~\ref{lem:Q4meta} and the sets $S_1^{d-4},S_2^{d-4},S_3^{d-4},S_4^{d-4}$ into Lemma~\ref{lem:meta}, we get that $m(Q_d;4)$ is at most
\begin{align*}
&\left(\frac{(d-4)((d-4)^2+3(d-4)+14)}{24}+1\right)+ 3\left(\frac{(d-4)(d-1)}{6}+1\right) + 3\left(\frac{d-4}{2}+1\right) + 1\\
&= \frac{d(d^2+3d+14)}{24}+1
\end{align*}
as desired.

\begin{casee}
$d\equiv 0\bmod 6$.
\end{casee}

By induction, there exists a percolating set $S_4^{d-12}$ of cardinality $\frac{(d-12)((d-12)^2+3(d-12)+14)}{24}+1$ for the $4$-neighbour bootstrap process in $Q_{d-12}$. Also, by Lemma~\ref{lem:r=3} and the fact that $d-12$ is divisible by $6$, there exist sets $S_1^{d-12}\subseteq S_2^{d-12}\subseteq S_3^{d-12}$ of vertices of $Q_{d-12}$ such that $|S_1^{d-12}|=1,|S_2^{d-12}|=\frac{d-12}{2}+1$ and $|S_3^{d-12}|=\frac{(d-12)(d-9)}{6}+1$ and $S_i^{d-12}$ percolates with respect to the $i$-neighbour bootstrap process on $Q_{d-12}$ for all $i\in\{1,2,3\}$. By feeding the labeling $\ell_0^{12}$ from Lemma~\ref{lem:Q12meta} and the sets $S_1^{d-12},S_2^{d-12},S_3^{d-12},S_4^{d-12}$ into Lemma~\ref{lem:meta}, we get that $m(Q_d;4)$ is at most
{\footnotesize
\begin{align*}
&\left(\frac{(d-12)((d-12)^2+3(d-12)+14)}{24}+1\right)+ 9\left(\frac{(d-12)(d-9)}{6}+1\right) + 33\left(\frac{d-12}{2}+1\right) + 55\\
&= \frac{d(d^2+3d+14)}{24}+1
\end{align*}}
as desired.

\end{proof}

\begin{proof}[Proof of Theorem~\ref{th:linearDiff}]
The cases $4\leq d\leq 15$ are covered by the lemmas in Section~\ref{sec:constructionsSeed4} and $d=16$ is covered by Theorem~\ref{th:main}. So, suppose that $d\geq17$. We give special attention to the case $d=17$. 

\begin{caseee}
$d=17$.
\end{caseee}

By Lemma~\ref{lem:Q13}, there exists a percolating set $S_4^{13}$ of cardinality $122$  for the $4$-neighbour bootstrap process in $Q_{13}$. Also, by Lemma~\ref{lem:r=3}, there exist sets $S_1^{13}\subseteq S_2^{13}\subseteq S_3^{13}$ of vertices of $Q_{13}$ such that $|S_1^{13}|=1,|S_2^{13}|=8$ and $|S_3^{13}|=36$ and $S_i^{13}$ percolates with respect to the $i$-neighbour bootstrap process on $Q_{13}$ for all $i\in\{1,2,3\}$. By feeding the labeling $\ell_0^{4}$ from Lemma~\ref{lem:Q4meta} and the sets $S_1^{13},S_2^{13},S_3^{13},S_4^{13}$ into Lemma~\ref{lem:meta}, we get $m(Q_{17};4)$ is at most
\[122 + 3\cdot 36 + 3\cdot 8+ 1 = 255< 272 = \left\lceil\frac{17(17^2+3\cdot 17+14)}{24}\right\rceil + 1 + 20\left\lfloor \frac{17-4}{12}\right\rfloor\]
and so the theorem is true in this case. 

\begin{caseee}
$d\geq18$.
\end{caseee}

By induction and the fact that $d-12\neq 5$, there exists a percolating set $S_4^{d-12}$ of cardinality $\left\lceil\frac{(d-12)((d-12)^2+3(d-12)+14)}{24}\right\rceil+1 +20\left\lfloor\frac{d-16}{12}\right\rfloor$  for the $4$-neighbour bootstrap process in $Q_{d-12}$. Also, by Lemma~\ref{lem:r=3}, there exist sets $S_1^{d-12}\subseteq S_2^{d-12}\subseteq S_3^{d-12}$ of vertices of $Q_{d-12}$ such that $|S_1^{d-12}|=1,|S_2^{d-12}|=\left\lceil\frac{d-12}{2}\right\rceil+1\leq \frac{d-12}{2}+\frac{1}{2}+1$ and $|S_3^{d-12}|=\left\lceil\frac{(d-12)(d-9)}{6}\right\rceil+1\leq \frac{(d-12)(d-9)}{6}+\frac{1}{3}+1$ and $S_i^{d-12}$ percolates with respect to the $i$-neighbour bootstrap process on $Q_{d-12}$ for all $i\in\{1,2,3\}$. By feeding the labeling $\ell_0^{12}$ from Lemma~\ref{lem:Q12meta} and the sets $S_1^{d-12},S_2^{d-12},S_3^{d-12},S_4^{d-12}$ into Lemma~\ref{lem:meta}, we get $m(Q_d;4)$ is at most
{\scriptsize
\begin{align*}
&\left(\left\lceil\frac{(d-12)((d-12)^2+3(d-12)+14)}{24}\right\rceil+1 + 20\left\lfloor\frac{d-16}{12}\right\rfloor\right)+ 9\left(\left\lceil\frac{(d-12)(d-9)}{6}\right\rceil+1\right) + 33\left(\left\lceil\frac{d-12}{2}\right\rceil+1\right) + 55\\
&=\Biggl\lceil\frac{(d-12)((d-12)^2+3(d-12)+14)}{24} + 9\left(\left\lceil\frac{(d-12)(d-9)}{6}\right\rceil+1\right) + 33\left(\left\lceil\frac{d-12}{2}\right\rceil+1\right) + 55\Biggr\rceil+1 + 20\left\lfloor\frac{d-16}{12}\right\rfloor\\
&\leq  \left\lceil\frac{d(d^2+3d+14)}{24} + 3 + \frac{33}{2}\right\rceil+1+20\left\lfloor\frac{d-16}{12}\right\rfloor\\
&\leq  \left\lceil\frac{d(d^2+3d+14)}{24}\right\rceil+\frac{81}{4}+1+20\left\lfloor\frac{d-16}{12}\right\rfloor
\end{align*}}
which, since $m(Q_d;4)$ is an integer, implies the result. 
\end{proof}

\section{Conclusion}

While we have only been able to obtain an exact result for some values of $d$, we have little doubt that Theorem~\ref{th:main} extends to all $d\geq4$, except for the case $d=5$ which was ruled out by exhaustive computer search by Morrison and Noel~\cite[Section~6]{MorrisonNoel18}.

\begin{conj}
For all $d\geq4$ with $d\neq5$, it holds that
\[m(Q_d; 4) = \left\lceil\frac{d(d^2+3d+14)}{24}\right\rceil+1.\]
\end{conj}

More generally, Morrison and Noel~\cite{MorrisonNoel18} ask whether, for every $r\geq1$, there exists $d_0(r)$ such that Theorem~\ref{th:MN} is tight for all $d\geq d_0(r)$. We believe that the answer is probably ``yes.'' However, verifying this for general $r$ will most likely require a new strategy for finding constructions which is not so heavily reliant on computer assistance. 

\begin{ack}
The author would like to thank the Google DeepMind AlphaEvolve team for the opportunity to participate as a Trusted Tester with early access to the system, and for their continued support throughout this project. The author is especially grateful to Adam Zsolt Wagner, who was the primary point of contact and provided extremely helpful guidance on the platform, and who ran a version of one of my experiments with extended CPU time. The author would also like to thank Natasha Morrison for valuable conversations related to the topic of this paper. 
\end{ack}

\bibliographystyle{plain}
\bibliography{bootstrap}
\end{document}